\documentclass[11pt,article]{amsart}
\usepackage{amssymb, latexsym, graphicx, mathrsfs, enumerate, amsmath, amsthm, textcomp, url, tikz-cd, bbm, mathtools, multirow, multicol}
\usepackage[T1]{fontenc}
\usepackage{mathdots}
\usepackage{comment}
\usepackage[toc,page]{appendix}
\usepackage[breaklinks=true]{hyperref}
\hypersetup{
    colorlinks=true,
    linkcolor=blue,
    filecolor=magenta,      
    urlcolor=cyan,
}

\tikzset{
    labl/.style={anchor=south, rotate=90, inner sep=.5mm}
}

\usepackage{color}

\input xy
\xyoption{all}

\allowdisplaybreaks
\allowdisplaybreaks

\numberwithin{equation}{section}
\newtheorem{theorem}{Theorem}[section]
\newtheorem{proposition}[theorem]{Proposition}
\newtheorem{corollary}[theorem]{Corollary}
\newtheorem{lemma}[theorem]{Lemma}

\newtheorem*{remark*}{Remark}

\theoremstyle{definition}
\newtheorem{definition}[theorem]{Definition}

\newtheorem{remark}[theorem]{Remark}

\newcommand{\C}{\mathbb{C}}

\newcommand{\Q}{\mathbb{Q}}

\newcommand{\R}{\mathbb{R}}
\newcommand{\Z}{\mathbb{Z}}

\makeatletter
\def\ops@declare#1{\expandafter\DeclareMathOperator\csname #1\endcsname{#1}}
\def\ops@scan#1,{\ifx#1\relax\let\ops@next\relax\else\ops@declare{#1}\let\ops@next\ops@scan\fi\ops@next}
\newcommand{\DeclareMathOperators}[1]{\ops@scan#1,\relax,}
\makeatother
\DeclareMathOperators{Hom,Nat,Tor,Ann,Aff,Aut,GL,SL,SO,Spec,lt,lp,lc,im,lcm,id,Set,Ab,Grp,Ring,Mod,Top,Vect,Ind,res,Gal,coker,sh,Supp,Br,diag,Stab,Pic,inv,Nm,ord,End}
\def\makebb#1{\expandafter\def\csname b#1\endcsname{{\mathbb{#1}}}\ignorespaces}
\def\makebf#1{\expandafter\def\csname bf#1\endcsname{{\mathbf{#1}}}\ignorespaces}
\def\makegr#1{\expandafter\def\csname f#1\endcsname{{\mathfrak{#1}}}\ignorespaces}
\def\makescr#1{\expandafter\def\csname s#1\endcsname{{\mathscr{#1}}}\ignorespaces}
\def\makec#1{\expandafter\def\csname c#1\endcsname{{\mathcal{#1}}}\ignorespaces}
\def\makecal#1{\expandafter\def\csname cal#1\endcsname{{\mathcal{#1}}}\ignorespaces}
\def\doLetters#1{#1A #1B #1C #1D #1E #1F #1G #1H #1I #1J #1K #1L #1M
	#1N #1O #1P #1Q #1R #1S #1T #1U #1V #1W #1X #1Y #1Z}
\def\doletters#1{#1a #1b #1c #1d #1e #1f #1g #1h #1j #1k #1l #1m
	#1n #1o #1p #1q #1r #1s #1t #1u #1v #1w #1x #1y #1z}
\doLetters\makebb
\doLetters\makecal
\doLetters\makec
\doLetters\makescr
\doLetters\makebf
\doletters\makebf
\doletters\makegr

\def\abs#1{\lvert#1\rvert}
\def\norm#1{\lVert#1\rVert}

\begin{document}

\title[An asymptotic formula for the number of integral matrices via orbital integrals]{An asymptotic formula for the number of integral matrices with a fixed characteristic polynomial via orbital integrals}



\author[Seongsu Jeon]{Seongsu Jeon}
\author[Yuchan Lee]{Yuchan Lee}
\thanks{The authors are supported by Samsung Science and Technology Foundation under Project Number SSTF-BA2001-04.}

\address{Seongsu Jeon \\  Department of Mathematics, POSTECH, 77, Cheongam-ro, Nam-gu, Pohang-si, Gyeongsangbuk-do, 37673, KOREA}

\email{ssjeon.math@gmail.com}

\address{Yuchan Lee \\  Department of Mathematics, POSTECH, 77, Cheongam-ro, Nam-gu, Pohang-si, Gyeongsangbuk-do, 37673, KOREA}

\email{yuchanlee329@gmail.com}

\begin{abstract}
For an irreducible polynomial $\chi(x)\in \mathcal{O}_k[x]$ of degree $n$, where $k$ is a number field and $\mathcal{O}_k$ its ring of integers,
let $N(X, T)$ denote the number of $n \times n$ integral matrices whose characteristic polynomial is $\chi(x)$, bounded by a positive real number $T$ with respect to a certain norm. 
In this paper, we provide an asymptotic formula for $N(X,T)$ as $T\to \infty$ in terms of the orbital integrals of $\mathfrak{gl}_n$.
This result extends the work of A. Eskin, S. Mozes, and N. Shah \cite{EMS} (1996) to a broader setting, thereby further developing the generalization initiated by the second author in \cite{Lee25}.

Our approach is based on the interpretation of local Brauer evaluations for $X$ via  local class field theory, and on the Langlands–Shelstad fundamental lemma for $\mathfrak{sl}_n$.
In particular, we observe that local Brauer evaluations for $X$ determine local endoscopic data for $\mathrm{SL}_n$, suggesting a deeper conceptual connection between these two notions.
% For an arbitrarily given irreducible polynomial $\chi(x) \in \mathbb{Z}[x]$ of degree $n$ with $\mathbb{Q}[x]/(\chi(x))$ totally real, let $N(X, T)$ be the number of $n \times n$  matrices over $\mathbb{Z}$ whose characteristic polynomial is $\chi(x)$, bounded by a positive number $T$ with respect to a certain norm. 
% In this paper, we will provide an asymptotic formula for $N(X, T)$ as $T \to \infty$ in terms of the orbital integrals of $\mathfrak{gl}_n$.
% This generalizes the work of A. Eskin, S. Mozes, and N. Shah (1996) which assumed that $\mathbb{Z}[x]/(\chi(x))$ is the ring of integers. 
% In addition, we will provide an asymptotic formula for $N(X, T)$, using the orbital integrals of $\mathfrak{gl}_n$, when $\mathbb{Q}$ is generalized to a totally real number field $k$ and when $n$ is a prime number.
% Here we need a mild restriction on splitness of $\chi(x)$ over $k_v$ at $p$-adic places $v$ of $k$ for $p \leq n$ when $k[x]/(\chi(x))$ is unramified Galois over $k$. 
% Our method is based on the strong approximation property with Brauer-Manin obstruction on a variety, the formula for orbital integrals of $\mathfrak{gl}_n$, and the Langlands-Shelstad fundamental lemma for $\mathfrak{sl}_n$.
\end{abstract}
\maketitle
\tableofcontents

\section{Introduction}
For an algebraic variety defined over $\mathbb{Z}$, understanding its $\mathbb{Z}$-points can be viewed as a classical number-theoretic problem, studying solutions to a Diophantine system.
More generally, one can expand objects to a variety $X$ defined over $\mathcal{O}_k$, where $\mathcal{O}_k$ is the ring of integers of a number field $k$.
In this context, a Diophantine analysis investigating the distribution of $X(\mathcal{O}_k)$
has been thoroughly studied for $k = \Q$ (\cite{Bir}, \cite{BR}, \cite{DRS}, and \cite{Sch}). 
Here the distribution of $X(\mathcal{O}_k)$ means the asymptotic behavior of $N(X,T)$ as $T\rightarrow \infty$ where 
\[
N(X,T)=\#\{x\in X(\mathcal{O}_k)\mid \norm{x}_\infty \leq T\}
\]with a certain norm $\norm{\cdot}_\infty$ defined on $X(k)$.

In this paper, we concentrate on a variety $X$ which represents the set of $n\times n$ matrices whose characteristic polynomial is $\chi(x)$, where $\chi(x)\in \mathcal{O}_k[x]$ is an irreducible monic polynomial of degree $n$.
We aim to formulate the function $A(T)$ such that $\frac{N(X,T)}{A(T)}\rightarrow 1$ as $T\rightarrow \infty$,
with respect to the norm given by 
\begin{equation}\label{def:norm}
\norm{x}_\infty = \max_{v\in \infty_k} \sqrt{\sum_{i,j = 1}^{n} |x_{ij}|_v^2},
\end{equation}
where $x$ embeds to $(x_{ij}) \in \mathrm{M}_n(k)$ and $\infty_k$ denotes the set of Archimedean places of $k$.

This is the case investigated by Eskin, Mozes, and Shah \cite{EMS} over $\mathbb{Q}$.
They \cite[Theorem 1.1]{EMS} proved the following formula:
    \begin{equation}\label{prop:EMS_intro}
            N(X,T)\sim \frac{2^{n-1}h_{K} R_{K} w_n}{\sqrt{\Delta_{\chi}}\cdot \prod_{k=2}^n\Lambda(k/2)}T^{\frac{n(n-1)}{2}}, 
        \end{equation}
under the assumptions that
\begin{equation}\label{eq:conditiononEMS}\left\{
\begin{array}{l}
\textit{(1) $k=\Q$}; \\
\textit{(2) $\mathbb{Z}[x]/(\chi(x))$ is the ring of integers of $\mathbb{Q}[x]/(\chi(x))$};\\ 
\textit{(3) $\mathbb{Q}[x]/(\chi(x))$ is totally real, equivalently $\chi(x)$ splits completely over $\R$}.
\end{array}\right.\end{equation}
Here, for two functions $A(T)$ and $B(T)$ for $T$, we write $A(T)\sim B(T)$ as $T\rightarrow \infty$ if $\lim\limits_{T\rightarrow \infty}\frac{A(T)}{B(T)}=1$.
\begin{itemize}
\item 
$h_{K}$ is the class number of $\mathcal{O}_K$ and $R_{K}$ is the regulator of $K$ for $K=k[x]/(\chi(x))$ and its ring of integers $\mathcal{O}_K$.
\item
$\Delta_{\chi}$ is the discriminant of $\chi(x)$, $w_n$ is the volume of the unit ball in $\mathbb{R}^{\frac{n(n-1)}{2}}$, and $\Lambda(s)=\pi^{-s}\Gamma(s)\zeta(2s)$.
\end{itemize}

In \cite[Theorem 6.7]{Lee25}, the second author extended this asymptotic formula for $N(X,T)$ to number field $k$ expressed in terms of orbital integrals on $\mathrm{GL}_n$.
%, thereby allowing the formula to be treated in a broader setting.
%The second author  \cite[Theorem 6.7]{Lee25} extended this framework to general number fields $k$, thereby broadening the scope of the original result.
% However, this extension was limited to the case where 
% $\chi(x)$ arises as the characteristic polynomial of $\mathrm{SL}_n(\mathcal{O}_k)$, that is, when $\chi(0)=1$, and still required further global assumptions: namely, that $k$ and $K$ are totally real, and that $\deg \chi(x)=n$ is prime when $k\neq \mathbb{Q}$.
% In other words, while \cite{Lee25} achieved a partial generalization, it remained constrained by the special form of $\chi(x)$.
This extension, however, was limited to the case where 
$\chi(x)$ arises as the characteristic polynomial of 
$\mathrm{SL}_n$, i.e. $\chi(0)=1$
% and required that $k$ and $K$ be totally real, with $\deg\ \chi(x)$  
\cite[Theorem 6.7]{Lee25}.
% In other words, while \cite{Lee25} provided a powerful generalization of \cite[Thoerem 1.1]{EMS}, the broader case of arbitrary characteristic polynomials remains open.

In this paper, we extend this result: we remove the assumption $\chi(0)=1$ when $\chi(x)$ splits over $k_v$ at $p$-adic places $v$ of $k$ for $p\leq n$, while retaining the assumptions on total realness of $k$ and $K$ and the prime degree of $\chi(x)$ as in \cite[Theorem 6.7]{Lee25}.
Consequently, our result extends the asymptotic formula of Eskin, Mozes, and Shah to this wider setting
\begin{equation}\label{intro:condition_ours}
\left\{
\begin{array}{l}
     \textit{(1) $k$ and $K\left(:=k[x]/(\chi(x))\right)$ are totally real number fields};  \\
     \textit{(2) if $k\neq \mathbb{Q}$, then  $\chi(x)$ is of prime degree $n$};\\
     \textit{(3) if $K/k$ is unramified Galois, then $\chi(0)=1$ or $\chi(x)$ splits over $k_v$ at $p$-adic places $v$ of $k$ for $p \leq n$},
\end{array}
\right.
\end{equation}
developing the extension initiated in \cite{Lee25}.

\cite{Lee25} approaches the case where $\chi(x)$ arises from $\mathrm{SL}_n$ by employing a global argument based on endoscopy theory.
More precisely, from a modified form of \cite[Theorem 4.3]{WX}, the author uses the pre-stabilization of Arthur's trace formula to derive the desired asymptotic.
Since directly applying this approach to the Lie algebra setting is not straightforward, we instead take a local perspective via \cite[Theorem 4.3]{WX} and work with explicit local integrals.

Under this approach, the most delicate part is the analysis of the Brauer evaluations appearing in local integrals.
In general, such computations are subtle.
In this work, we carry this out using local class field theory and, furthermore, we observe that Brauer evaluations determine local endoscopic data for $\mathrm{SL}_n$ (see Lemma \ref{lem:eval-endo}).
This suggests a broader connection between Brauer evaluations and endoscopic data, which we hope will be clarified in future work.

\subsection{Backgrounds}
One of the important observations for $k = \Q$ is introduced in \cite{BR}, which is called Hardy-Littlewood expectation.
Borovoi and Rudnick \cite[Theorem 5.3]{BR} proved that a particular homogeneous space of a simply connected semisimple group, including our case when $k=\mathbb{Q}$, can be represented by the integration of a certain function defined on the adelic points of $X$.
% However, their description of the integrand is complicated to compute directly (cf. \cite[Section 3.5, Theorem 5.3]{BR}).

Wei and Xu \cite{WX} generalized the observation of Borovoi and Rudnick for $X$ defined over any number field $k$.
They mainly used the fact that $X$ satisfies the strong approximation property with Brauer-Manin obstruction following \cite[Theorem 0.1]{BD}.
% This implies that integral points of $X$ are related to the Brauer elements (cf. \cite[Theorem 2.10]{LX}).
% Combining this observation with the following equidistribution property \cite[(4.2)]{WX}, they obtained Proposition \ref{prop:main_idea_intro}.
% \begin{equation}\label{eq:equidistribution_property}
% |\{y\in x\cdot \Gamma\mid \norm{y}_{\infty}\leq T \}|\sim
% \frac{m^{H_x}_{\infty}(\Gamma\cap H_x(k)\backslash H_x(k_\infty) )}{m^G_{\infty}(\Gamma \backslash G(k_\infty))}
% m_{\infty}^X(x\cdot G(k_\infty)\cap X(k_\infty,T)),
% \end{equation} where $H_x$ is the stabilizer of $x$ in $G$. 
% Here, $m_\infty^G$, $m_\infty^{H_x}$, and $m_\infty^X$ denote the infinite components of Tamagawa measures on $G(\mathbb{A}_k)$, $H(\mathbb{A}_k)$, and $X(\mathbb{A}_k)$, respectively. 
To introduce the results in \cite{WX}, we use the following notations:
\[
\left\{
\begin{array}{l}
     \textit{$\Omega_k$ is the set of places of $k$, and $\infty_k$ is the set of Archimedean places of $k$};\\
     \textit{$k_v$ is a completion with respect to the place $v$}; \\
     \textit{$\mathcal{O}_{k_v}$ is its ring of integers and $\pi_v$ is a uniformizer of $\mathcal{O}_{k_v}$ for $v\in \Omega_k\setminus \infty_k$};\\
     \textit{$\kappa_v$ is the residue field of $\mathcal{O}_{k_v}$ and $q_v=\#\kappa_v$ for $v\in \Omega_k\setminus \infty_k$}.
     % \textit{For an element $x\in k_v$, the exponential order of $x$ with respect to the maximal ideal in $\mathcal{O}_{k_v}$ is written by $\ord_v(x)$}.
\end{array}
 \right.
\]

\begin{proposition}{\cite[Theorem 4.3]{WX}}\label{prop:main_idea_intro}
The following asymptotic equivalence holds
    \begin{equation}\label{eq:main_idea_intro}
        N(X, T) \sim \sum_{\xi \in \Br X / \Br k} \prod_{v \in \Omega_k \setminus\infty_k} \int_{X(\mathcal{O}_{k_v})} \xi_v(x) \ m^X_v \prod_{v \in \infty_k} \int_{X(k_v, T)} \xi_v(x) \ m^X_v, \end{equation}
        where
        \[
        \left\{
        \begin{array}{l}
        X(k_v, T)= \{x\in X(k_v)\mid \norm{x}_{v}\leq T\} \textit{ for $v\in \infty_k$ and $T > 0$};  \\
        m^X=\prod_{v\in\Omega_k}m^X_v \textit{ is the Tamagawa measure on $X(\mathbb{A}_k)$ }\text{(see Section \ref{sec:Tammeasure})}; \\
        \Delta_k \textit{ is the absolute discriminant of $k$}; \\
        \textit{$\Br X / \Br k$ is the finite abelian group defined in Remark \ref{rem:normalized_eval}}; \\
        \xi_v(x)\textit{ denotes a Brauer evaluation of $\xi$ at $x \in X(k_v)$ (see Definition \ref{def:brauer_evaluation})}.
    \end{array}\right.
    \]
\end{proposition}\noindent
Via this result, they deduced the same results with (\ref{prop:EMS_intro}).
%under the assumptions in (\ref{eq:conditiononEMS}).
They showed that for a nontrivial element $\xi\in\Br X/\Br k$,
$\int_{X(\Z_p)} \xi_p(x) \ m^X_p=0$ for a ramified prime number $p$.
Hence the corresponding summand on the right-hand side of (\ref{eq:main_idea_intro}) vanishes.
Then they computed the local integral for a trivial element in $\Br X / \Br k$, under the assumption (\ref{eq:conditiononEMS}). 

However, their method is not applicable for a general number field since the ramified place might not exist for a number field other than $\mathbb{Q}$ (see \cite[Example 6.3]{WX}).
Moreover, computing the local integral is challenging unless $\Z[x]/(\chi(x))$ coincides with the ring of integers of $\mathbb{Q}[x]/(\chi(x))$
(see Proposition \ref{prop:result_case:trivial}).

\begin{remark}\label{rmk:equidistribution_for_SLn}
    In fact, \cite[Theorem 4.3]{WX} is stated under the assumption that $X$ satisfies the equidistribution property  (4.2) in loc. cit.
    As shown in \cite[Remark 6.1]{Lee25}, this assumption is indeed fulfilled for $X$ in our setting.
\end{remark}

\subsection{Main results}\label{intro:main}
Our strategy to obtain the asymptotic formula for $N(X, T)$ is based on Proposition \ref{prop:main_idea_intro}.
The main observation is to describe the evaluation of $\xi \in \Br X$ in (\ref{eq:main_idea_intro}) by means of several commutative diagrams and local class field theory.

According to Section \ref{sec:22}, our method depends on the $\mathrm{GL}_n$-homogeneous space (resp. $\mathrm{SL}_n$-homogeneous space) structure of $X$.
As in Proposition \ref{prop:homogeneous}, we fix $x_0\in X(\mathcal{O}_k)$ such that
\[
    \mathrm{T}_{x_0}\backslash \mathrm{GL_n} \xrightarrow{\sim} X \ (resp.\ \mathrm{S}_{x_0}\backslash \mathrm{SL}_n \xrightarrow{\sim} X),\ g\mapsto g^{-1}x_0g,
\]
where $\mathrm{T}_{x_0}$ (resp. $\mathrm{S}_{x_0}$) denotes the centralizer of $x_0$ under the $\mathrm{GL}_n$-conjugation (resp. the $\mathrm{SL}_n$-conjugation) on $X$.

\subsubsection{\textbf{Brauer evaluation of $\xi \in \Br X$ on $X(k_v)$}}
The Brauer evaluation of $\xi \in \Br X$ on $X(k_v)$ is defined by using the functoriality of the Brauer group in Definition \ref{def:brauer_evaluation}.
% We interpret this functoriality through the local class field theory and obtain Proposition \ref{prop:evaluation_intro}.
To establish the well-definedness of the local evaluation of
$\xi \in \Br X$, we define the normalization $\tilde{\xi}_v$ in Definition \ref{def:normalizedeval}. 
In (\ref{eq:normalized_evaluation}), we show that this normalization does not affect the product in (\ref{eq:main_idea_intro}).

\begin{proposition}\label{prop:evaluation_intro}
    Suppose that $\chi(x)$ is of prime degree. For $\xi \in \Br X$, the normalized evaluation $\tilde{\xi}_v$ is formulated as follows.
    \begin{enumerate} 
        \item (Proposition \ref{prop:not_galois}) If $K/k$ is not Galois, then $\tilde{\xi}_v(x) = 1$. 
        \item (Proposition \ref{prop:evaluation}.(2)) If $K/k$ is Galois and $K_v/k_v$ is not a field extension, then $\tilde{\xi}_v(x) = 1$. 
        \item (Proposition \ref{prop:evaluation}.(1)) If $K/k$ is Galois and $K_v/k_v$ is a field extension, then 
        \[\tilde{\xi}_v(x) = \exp 2\pi i \Psi(\xi) (\phi_{K_v/k_v} (\det g_x)),\]
        where
        \[
    \left\{
    \begin{array}{l}
    \textit{$\Psi : \Br X / \Br k \xrightarrow{\sim} \Hom(\Gal(K/k), \Q/\Z)$ is an isomorphism defined in Proposition \ref{prop:evaluation}};\\
    \textit{$g_x \in \mathrm{GL}_n(k_v)$ such that $g_x^{-1} x_0 g_x = x$};\\ 
    \textit{$\phi_{K_v/k_v} : k_v^\times / \Nm_{K_v/k_v} K_v^\times \xrightarrow{\sim} \Gal(K_v/k_v)$ is the Artin reciprocity isomorphism }
    \\ (\textit{cf. \cite[Chapter XI]{Ser}}).
    \end{array}
    \right.
    \]
    \end{enumerate}
\end{proposition}\noindent

\subsubsection{\textbf{Computation of local integrals}}\label{sec122}
Now we formulate the integral of $\tilde{\xi}_v$ for each $v \in \Omega_k$ using the more manageable measure $d\mu_v$ in (\ref{eq:quotient_measure}).
Here, Lemma \ref{lem:measure_comparison_m_mu} allows us to replace the measure $m^X$ with $\prod_{v \in \Omega_k} d\mu_v$.

In the case of an Archimedean place, it is enough to consider the case that $\tilde{\xi}_v$ is trivial by Proposition \ref{prop:evaluation_intro}.(2) (see assumption (1) in (\ref{intro:condition_ours})).
Then the computation of the integral directly follows from \cite[Lemma A.1]{Lee25}.
\begin{proposition}\label{prop:int_arch}(Proposition \ref{prop:eq_int_arch})
     Suppose that $k$ and $K$ are totally real number fields. Then we have
\[ 
\prod_{v\in \infty_k} \int_{X(k_v,T)}\tilde{\xi}_v(x) \ d\mu_v \sim 
\left( \frac{w_n\pi^{\frac{n(n+1)}{4}} }{\prod_{i=1}^n \Gamma(\frac{i}{2})} T^{\frac{n(n-1)}{2}}  \right)^{[k:\Q]}\prod_{v\in\infty_k}|\Delta_\chi|_v^{-\frac{1}{2}},\]
where $w_n$ is the volume of the unit ball in $\mathbb{R}^{\frac{n(n-1)}{2}}$ and $\Delta_\chi$ is the discriminant of $\chi(x)$.
\end{proposition}

In the case of a non-Archimedean place, 
Proposition \ref{prop:evaluation_intro} implies that $\tilde{\xi}_v$ is non-trivial only if $K/k$ is Galois and $K_v/k_v$ is a field extension.
In this case, we address the following cases separately; when $K_v/k_v$ is ramified and when $K_v/k_v$ is unramified.

\begin{proposition}\label{prop:int_nonarch}
Suppose that $v\in \Omega_k \setminus \infty_k$.
    \begin{enumerate}
        \item  (Proposition \ref{prop:result_case:trivial})
        If $\tilde{\xi}_v$ is trivial, then we have
    \[\int_{X(\mathcal{O}_{k_v})}\tilde{\xi}_v(x)\ d\mu_v=\int_{\mathrm{T}_{x_0}(k_v)\backslash \mathrm{GL}_n(k_v)}\mathbbm{1}_{\mathfrak{gl}_n(\mathcal{O}_{k_v})}(g^{-1}x_0g)\ d\mu_v.
    \]
     Here, the right hand side is called the orbital integral of $\mathfrak{gl}_n$ for $\mathbbm{1}_{\mathfrak{gl}_n(\mathcal{O}_{k_v})}$ and the characteristic polynomial $\chi(x)$ with respect to the measure $d\mu_v$ (see \cite[Section 1.3]{Yun13}), denoted by $\mathcal{O}_{\chi,d\mu_v}(\mathbbm{1}_{\mathfrak{gl}_{n}(\mathcal{O}_{k_v})})$

    \item If $\tilde{\xi}_v$ is non-trivial (and hence $K_v/k_v$ is a Galois field extension), then we have the following equations.
    \begin{itemize}
        \item (Proposition \ref{prop:ramify_zero}) If $K_v/k_v$ is a ramified Galois extension, then we have 
    \[
    \int_{X(\mathcal{O}_{k_v})}\tilde{\xi}_v(x) \ d\mu_v =0.
    \]
    \item (Proposition \ref{prop:fundlemforsln}) If $K_v/k_v$ is an unramified Galois extension and the characteristic of the residue field of $k_v$ is bigger than $n$, then we have
   \[
   \int_{X(\mathcal{O}_{k_v})} \tilde{\xi}_v(x)\ d\mu_v
   =\abs{\Delta_{\chi}}_v^{-\frac{1}{2}},
   \]
   where $\Delta_\chi$ is the discriminant of $\chi(x)$.
    \end{itemize}
    \end{enumerate}
\end{proposition}
% The computation in Proposition \ref{prop:int_nonarch} is similar with that of \cite[Section 6.2]{Lee25} but it contains significant distinction.
% Other than \cite[Section 6.2]{Lee25} which uses the global arguments with endoscopy theory, we directly interpret the local evaluation $\tilde{\xi}_v$, so that our computation is local nature.
% The most important difference to \cite[Section 6.2]{Lee25} is that, in our case, $x_0$ might not be in $\mathrm{SL}_n$ so that we use the orbital integral for Lie algebra $\mathfrak{gl}_n$.
% In the case that $\tilde{\xi_v}$ is non-trivial and $K_v/k_v$ is ramified, the computation is completely independent of \cite[Section 6.2]{Lee25}, we extend the method in \cite[Theorem 6.1]{WX}.
% On the other hand, in the case that $\tilde{\xi}$ is non-trivial and $K_v/k_v$ is unramified, we also use the fundamental lemma as in \cite[Section 6.2]{Lee25} but for Lie algebra (see \cite[Theorem 1]{Ngo}). Moreover, in our case, the endoscopic structure is hidden in the Brauer evlauation so that we need to deduce that $\tilde{\xi}_v$ assigns an endoscopic data of $\mathrm{SL}_n$ (see Lemmas \ref{lem:charstofevalxi}-\ref{lem:endo}) then this data would be corresponds to the endoscopic pair in \cite[Section 7.1]{Kot84} rather than the endoscopic triple given in \cite[Section 3.3]{Kal24} and \cite[Section 4.1.2]{Lee25}.
Other than \cite[Section 6.2]{Lee25}, in our case, $x_0$ may not lie in $\mathrm{SL}_n(\mathcal{O}_k)$.
Hence, we work with local orbital integrals of the Lie algebra $\mathfrak{gl}_n$.
% which employs global arguments based on endoscopy theory, we directly interpret the local evaluation $\tilde{\xi}_v$.
%, so that our computation is of local nature.
% Moreover, the most significant difference is that,  
When $\tilde{\xi}_v$ is nontrivial and $K_v/k_v$ is ramified, 
our computation is completely independent of \cite[Section 6.2]{Lee25}; in this case, we extend the method used in the proof of \cite[Theorem 6.1]{WX}.
On the other hand, when $\tilde{\xi}_v$ is nontrivial and $K_v/k_v$ is unramified, we apply the fundamental lemma as in \cite[Section 6.2]{Lee25}, but for the Lie algebra setting (see \cite[Theorem 1]{Ngo}).
In our situation, the endoscopic data is implicitly given from the Brauer evaluation $\tilde{\xi}_v$: 
%and Hence, we must first show that $\tilde{\xi}_v$ gives rise to an endoscopic datum for $\mathrm{SL}_n$ 
% (see Lemmas \ref{lem:charstofevalxi}–\ref{lem:endo}).
\begin{lemma}[Lemmas \ref{lem:charstofevalxi}–\ref{lem:endo}] \label{lem:eval-endo}
    % character $\kappa_{\tilde{\xi}_v}$ of $\mathrm{H}^1(k_v,\mathrm{S}_{x_0,k_v})$, which corresponds to the 
    The normalized evaluation $\tilde{\xi}_v$ determines an endoscopic data $(\kappa_{\tilde{\xi}_v}, \sigma_{\mathrm{S}_{x_0,k_v}})$ for $\mathrm{SL}_n$ (see Definition \ref{def:endo}), whose associated endoscopic group is $\mathrm{S}_{x_0,k_v}$. 
    Here, $\kappa_{\tilde{\xi}_v}$ is an element of $\mathrm{Hom}(X_*(\mathrm{S}_{x_0,k_v}),\mathbb{C}^\times)$, coming from the chracter of $\mathrm{H}^1(k_v,\mathrm{S}_{x_0,k_v})$ determined by $\tilde{\xi}_v$ (see the discussion following Lemma \ref{lem:charstofevalxi}), and
    $\sigma_{\mathrm{S}_{x_0,k_v}}$ denotes the automorphism on $X^*(\mathrm{S}_{x_0,k_v})$ induced by the Frobenius automorphism on $\mathrm{S}_{x_0,k_v}$.
\end{lemma}\noindent
We note that this datum is given by the pair described in \cite[Section 7.1]{Kot84}, rather than by the triple treated in \cite[Section 3.3]{Kal} and \cite[Section 4.1.2]{Lee25}.
This lemma suggests the existence of a more general connection between Brauer evaluations and the endoscopic data.

% Although the orbital integral, which appears in the geometric side of the Arthur-Selberg trace formula, has been studied extensively in the literature, to obtain its closed formula is quite involved.

% In this paper, for the closed formula of the orbital integral for $\mathfrak{gl}_n$, we will refer to the results in \cite{CKL} for the case that $n=2,3$. 

% In the case that $\tilde{\xi}_v$ is non-trivial and $K_v/k_v$ is an unramified Galois extension, we use the Langlands-Shelstad fundamental lemma in \cite[Theorem 1]{Ngo}.
% This is the reason why we need  the assumption (3) in (\ref{intro:condition_ours}), involving the characteristic of $\kappa_v$.

\subsubsection{\textbf{Conclusion}}
Combining Propositions \ref{prop:evaluation_intro}-\ref{prop:int_nonarch}, we conclude our main theorem. 
\begin{theorem}[Theorem \ref{thm:main_thm}]\label{thm:intro_main_thm}
Let  $\chi(x) \in \mathcal{O}_k[x]$ be an irreducible monic polynomial of degree $n$.
Let $X$ be an $\mathcal{O}_k$-scheme representing the set of $n\times n$ matrices whose characteristic polynomial is $\chi(x)$.
We define \[N(X, T)=\#\{x\in X(\mathcal{O}_k)\mid \norm{x}\leq T\},\]for $T>0$, where the norm $\norm{\cdot}$ is defined in (\ref{def:norm}).

Suppose that $k$ and $K = k[x]/(\chi(x))$ are totally real, and
if $k \neq \mathbb{Q}$, we further assume that $n$ is a prime number. 
We then have the following asymptotic formulas.
\begin{enumerate}
        \item If $K/k$ is not Galois or ramified Galois, then
        \[ N(X, T) \sim C_T \prod_{v \in \Omega_k \setminus\infty_k } \frac{\mathcal{O}_{\chi,d\mu_v}(\mathbbm{1}_{\mathfrak{gl}_n(\mathcal{O}_{k_v})})}{q_v^{S_v(\chi)}}.\]
    
        \item If $K/k$ is unramified Galois, and $\chi(0)=1$ or $K_v$ splits over $k_v$ for all $p$-adic places for $p \leq n$, then
        \[N(X, T) \sim C_T \left(  \prod_{v \in \Omega_k \setminus\infty_k }\frac{\mathcal{O}_{\chi,d\mu_v}(\mathbbm{1}_{\mathfrak{gl}_n(\mathcal{O}_{k_v})})}{q_v^{S_v(\chi)}} +n-1 \right).\]        
\end{enumerate}
Here, $\mathcal{O}_{\chi,d\mu_v}(\mathbbm{1}_{\mathfrak{gl}_n(\mathcal{O}_{k_v})})$ denotes the orbital integral of $\mathfrak{gl}_n$ for $\mathbbm{1}_{\mathfrak{gl}_n(\mathcal{O}_{k_v})}$ and the characteristic polynomial $\chi(x)$ with respect to the measure $d\mu_v$ defined in (\ref{eq:quotient_measure}), $S_v(\chi)$ is the $\mathcal{O}_{k_v}$-module length between $\mathcal{O}_{K_v}$ and $\mathcal{O}_{k_v}[x]/(\chi(x))$, 
    \[C_T:= |\Delta_k|^{\frac{-n^2+n}{2}} \frac{R_K h_K \sqrt{\abs{\Delta_K}}^{-1}}{R_k h_k \sqrt{\abs{\Delta_k}}^{-1}} \left(\prod_{i=2}^n \zeta_k(i)^{-1}\right)\left( \frac{2^{n-1}w_n\pi^{\frac{n(n+1)}{4}} }{\prod_{i=1}^n \Gamma(\frac{i}{2})} T^{\frac{n(n-1)}{2}}  \right)^{[k:\Q]},\]
    and we use the following notations:
    \begin{itemize}
        \item $R_F$ is the regulator of $F$, $h_F$ is the class number of $\mathcal{O}_F$, and $\Delta_F$ is the discriminant of $F/\mathbb{Q}$ for $F=k$ or $K$. 
        \item $w_n$ is the volume of the unit ball in $\mathbb{R}^{\frac{n(n-1)}{2}}$, and $\zeta_k$ is the Dedekind zeta function of $k$.
    \end{itemize}
    % Here, $C_T$ is formulated as follows
    % \[C_T = \abs{\Delta_k}^{\frac{-n^2+n}{2}} \frac{R_K h_K \sqrt{\abs{\Delta_K}}^{-1}}{R_k h_k \sqrt{\abs{\Delta_k}}^{-1}} \prod_{i=2}^n \zeta_k(i)^{-1} \left( \frac{ 2^{n-1} \pi^{\frac{n(n+1)}{4}} w_n}{\prod_{i=1}^n \Gamma(\frac{i}{2})} T^{\frac{n(n-1)}{2}}  \right)^{[k:\Q]},\]
    % and we use the following notations
    % \[
    % \left\{
    % \begin{array}{l}
    %  \textit{$R_F$ is the regulator of $F$ for $F = k$ or $K$};\\
    %   \textit{$h_F$ is the class number of $\mathcal{O}_F$ for $F = k$ or $K$};\\
    %   \textit{$\Delta_F$ is the discriminant of $F/\mathbb{Q}$ for $F = k$ or $K$};\\
    %   \textit{$w_n$ is the volume of the unit ball in $\mathbb{R}^{\frac{n(n-1)}{2}}$};\\
    %   \textit{$\zeta_k$ is the dedekind zeta function of $k$};\\
    %   \textit{$S_v(\chi)$ is the $\mathcal{O}_{k_v}$-module length between $\mathcal{O}_{K_v}$ and $\mathcal{O}_{k_v}[x]/(\chi(x))$};\\
    %   \textit{$\mathcal{O}_{\chi,d\mu_v}(\mathbbm{1}_{\mathfrak{gl}_n(\mathcal{O}_{k_v})})$ is the orbital integral for $\mathfrak{gl}_{n,k_v}$ associated with $\chi(x)$ with respect to $\frac{dg_v}{dt_v}$} \\ 
    %   \textit{(cf. Proposition \ref{prop:result_case:trivial} and Lemma \ref{cor:compareclassic})}.
    % \end{array}
    % \right.
    % \]
\end{theorem}\noindent
    This result is identical to \cite[Theorem 6.7]{Lee25} when $\chi(0)=1$. Indeed, when $x_0 \in \mathrm{SL}_n(\mathcal{O}_k)$, we can replace $\mathbbm{1}_{\mathfrak{gl}_n(\mathcal{O}_{k_v})}$ with $\mathbbm{1}_{\mathrm{GL}_n(\mathcal{O}_{k_v})}$ in the formulas of Theorem \ref{thm:main_thm}.
    % Therefore it refines the result in \cite[Theorem 6.7]{Lee25}.
In Remark \ref{rmk:our_result_EMS}, we show that the above formula coincides with (\ref{prop:EMS_intro}) under the assumptions in (\ref{eq:conditiononEMS}).

\begin{remark}\label{rmk:SOfinite}
    According to \cite[Theorem 1.5]{Yun13}, $\mathcal{O}_{\chi,d\mu_v}(\mathbbm{1}_{\mathfrak{gl}_n(\mathcal{O}_{k_v})})$ is an integer.
    %, whose leading term is $q_v^{S_v(\chi)}$.
    Since $S_v(\chi)=0$ and $\mathcal{O}_{\chi,d\mu_v}(\mathbbm{1}_{\mathfrak{gl}_n(\mathcal{O}_{k_v})})=1$ for all but finitely many places $v\in \Omega_k\setminus \infty_k$, the product \[\prod_{v \in \Omega_k \setminus\infty_k } \frac{\mathcal{O}_{\chi,d\mu_v}(\mathbbm{1}_{\mathfrak{gl}_n(\mathcal{O}_{k_v})})}{q_v^{S_v(\chi)}}\] is a finite product.

    In particular, when $n=2$ or $n=3$, the closed formula for the product $\displaystyle\prod_{v\in \Omega_k\setminus \infty_k}\mathcal{O}_{\chi,d\mu_v}(\mathbbm{1}_{\mathfrak{gl}_n(\mathcal{O}_{k_v})})$ is given by \cite{CKL} (see Proposition \ref{prop:productoforbitalintegral}). 
\end{remark}
% \begin{remark}\label{rmk_in_intro}
% In this remark, we explain how the assumptions in (\ref{intro:condition_ours}) affect the main steps of our argument.
%     \begin{enumerate}
%         \item 
%         Since we assume that $k$ is totally real, for the integration on $X(k_v,T)$ where $v\in\infty_k$, we can apply the theory of $\mathbb{R}$-manifolds (cf. Propsition \ref{prop:eq_int_arch}-\ref{lem:ldu_decompose}).
        
%         On the other hand, by virtue of the assumption that $K$ is totally real, Proposition \ref{cor:allxivaluation} yields that the Brauer evaluations for Archimedean places are trivial.
%         Moreover, this condition enables us to describe the stabilizer $\mathrm{T}_{x_0, k_v}$, for $v\in \infty_k$, in Section \ref{sec:measure}. 
%         This description affects the computations in Proposition \ref{sec:infinite}.

%         \item The assumption (2) in (\ref{intro:condition_ours}) implies that the Brauer evaluation is non-trivial only if $K/k$ is an abelian extension, as explained in Proposition \ref{prop:not_galois}, which enables us to use Proposion \ref{prop:evaluation}.
        
%         This assumption also facilitates figuring out an endoscopic group associated with a local integration for $k_v$, in Lemma \ref{lem:endo}, when $K_v/k_v$ is unramified.

%         \item
%         As we explained in Section \ref{sec122}, we assumed (3) in (\ref{intro:condition_ours}) in order to use the Langlands-Shelstad fundamental fundamental lemma for the Lie algebra $\mathfrak{sl}_n$ of $\mathrm{SL}_n$ over $k_v$, where $K_v/k_v$ is unramified.
%     \end{enumerate}
% \end{remark}
\vspace{1em}

\textbf{Organizations.}
In Section \ref{sec:22}, we review the homogeneous space structures of $\mathrm{GL}_n$ and $\mathrm{SL}_n$ on $X$, and describe $\mathrm{GL}_n(k_v)$- and $\mathrm{SL}_n(k_v)$-orbits on $X(k_v)$ for a local field $k_v$.
We also introduce the measures used in this paper, including the Tamagawa measure, in Section \ref{sec:22}.
In Section \ref{sec:Brauer}, we recall the notions of Brauer group and Braure evaluation.
To obtain the asymptotic formula for $N(X,T)$, we make use of Theorem \ref{eq:main_idea_intro}.
For this, we compute the Brauer evaluation $\xi_v$ on $X$ in Section \ref{sec:5} and the corresponding local integrals in Section \ref{sec:integral_computation}.
Combining these computations, we deduce in Section \ref{sec:main} the desired formula expressed in terms of orbital integrals of $\mathfrak{gl}_n$.
In Appendix \ref{app:orbital_integral}, we provide the product of local orbital integrals when $n=2$ and $3$, following \cite{CKL}.
\vspace{1em}

\textbf{Acknowledgments.}
We would like to thank our advisor Sungmun Cho for suggesting this problem and encouraging us, and Sug Woo Shin for helpful comments on Section \ref{sec:unram_cal}.
We also thank Seewoo Lee for meaningful discussions on the exterior product of differential forms
and Fei Xu for his lecture note on Brauer-Manin obstruction, which was given at POSTECH in 2019.
\subsection{Notations}
\begin{itemize}
\item Let $k$ be a number field with ring of integers $\mathcal{O}_k$, and $\bar{k}$ be an algebraic closure of $k$.
\item
Suppose that $A$ is a commutative $\mathcal{O}_k$-algebra. We define the following schemes over $A$.
\[\left\{\begin{array}{l}
\textit{$\mathrm{GL}_{n, A}:$ the general linear group scheme over $A$, $\mathfrak{gl}_{n, A}:$ the Lie algebra of $\mathrm{GL}_{n, A}$};\\
\textit{$\mathrm{SL}_{n, A}:$ the special linear group scheme over $A$, $\mathfrak{sl}_{n, A}:$ the Lie algebra of $\mathrm{SL}_{n, A}$};\\
\textit{$\mathrm{M}_{n,A}:$ the scheme over $A$ representing the set of $n \times n$ matrices.}
\end{array}
\right.\]
If there is no confusion, we sometimes omit $A$ in the subscript to express schemes over $k$. 

\item
Let $\chi(x) \in \mathcal{O}_k[x]$ be an irreducible monic polynomial of degree $n$.
We define $X$ to be the closed subscheme of $\mathrm{M}_{n,\mathcal{O}_k}$ representing the set of $n \times n$ matrices whose characteristic polynomial is $\chi(x)$.

\item Let $\Omega_k$ be the set of places and $\infty_k$ be the set of Archimedean places.
For $v\in \Omega_k$, we denote the normalized absolute value associated with $v$ by $|\cdot|_v$ and the completion with respect to $|\cdot|_v$ by $k_v$.

\item 
For $v \in \Omega_k \setminus \infty_k$, we define the following notations
\[\left\{\begin{array}{l}
\mathcal{O}_{k_v}:\textit{ the ring of integers of $k_v$};\\
\pi_v:\textit{ a uniformizer of $\mathcal{O}_{k_v}$};\\
\kappa_v:\textit{ the residue field of $\mathcal{O}_{k_v}$};\\
q_v:\textit{ the cardinality of the residue field $\kappa_v$}.
\end{array}
\right.\]
For an element $x\in k_v$, the exponential order of $x$ with respect to the maximal ideal in $\mathcal{O}_{k_v}$ is written by $\ord_v(x)$. We then have that $|x|_v=q_v^{-\ord_v(x)}$.

\item
We define $K= k[x]/(\chi(x))$ with ring of integers $\mathcal{O}_K$.
For $v \in \Omega_k$, we define  $K_v=k_v[x]/(\chi(x))$ by $K_v$ and, for $v\in\Omega_k\setminus \infty_k$, denote its ring of integers by $\mathcal{O}_{K_v}$.

\item
For $v \in \Omega_k$, let $B_v(\chi)$ be an index set in bijection with the irreducible factors $\chi_{v,i}$ of $\chi$ over $k_v$.
For $i \in B_v(\chi)$, we define the following notations
\[
\left\{
\begin{array}{l}
    \textit{$K_{v,i} = k_v[x]/(\chi_{v,i}(x))$ with ring of integers $\mathcal{O}_{K_{v,i}}$};\\
    \kappa_{K_{v,i}}:\textit{ the residue field of $K_{v,i}$}.
\end{array}
\right.
\]
Here, we note that $K_v \cong \prod_{i \in B_v(\chi)} K_{v,i}$ and $\mathcal{O}_{K_v} \cong \prod_{i \in B_v(\chi)}\mathcal{O}_{K_{v,i}}$.

\item
For a finite field extension $E/F$ where $F$ is $\Q$ or $k_v$, let $\Delta_{E/F}$ be the discriminant ideal of $E/F$.
\begin{itemize}
    \item 
If $F = \Q$, then we denote $\Delta_{E/\Q}$ by $\Delta_E$, and $\abs{\Delta_{E}}$ by the absolute value of a genearator of $\Delta_{E}$ as an ideal in $\mathbb{Z}$.
\item
If $F = k_v$, then $\ord_v(\Delta_{E/k_v})$ (resp. $|\Delta_{E/k_v}|_v$) denotes the exponential order (resp. the normalized absolute value) of a generator of $\Delta_{E/k_v}$ as an ideal in $\mathcal{O}_{k_v}$.
\end{itemize}

\item 
For a field $F$, let $\Delta_f \in F$ be the discriminant of a polynomial $f(x) \in F[x]$.

\item
We define norms for $x \in X(k)$ by
\[\norm{x}_v = \sqrt{\sum_{i,j = 1}^{n} |x_{ij}|_v^2} \ \  \text{ and } \ \norm{x}_\infty = \max_{v\in \infty_k} \norm{x}_v,\]
where $x$ embeds to $(x_{ij}) \in \mathrm{M}_n(k)$.
For $T > 0$, we define
\[N(X,T) = \# \{x \in X(\mathcal{O}_k) \mid \norm{x}_\infty\leq T\}.\]
We note that $N(X, T)$ is finite since $\mathcal{O}_k$ is discrete in $k_v$ for each Archimendian place $v\in \infty_k$.

\item 
For two functions $A(T)$ and $B(T)$ for $T$, we denote $A(T)\sim B(T)$ if $\lim\limits_{T\rightarrow \infty}\frac{A(T)}{B(T)}=1$.
\end{itemize}

\section{Homogeneous space \texorpdfstring{$X$}{X}}\label{sec:22}

\subsection{Homogeneous space}\label{subsec:homogeneous}
$X$ is a homogeneous space of both $\mathrm{GL}_{n}$ and $\mathrm{SL}_{n}$.
% Write $\chi(x)=x^n + \sum_{i=0}^{n-1} a_i x^i$ where $a_i \in \mathcal{O}_k$ for $0\leq i \leq n-1$.
We fix an integral matrix $x_0 \in X(\mathcal{O}_k)$ whose characteristic polynomial is $\chi(x)$. 
For example, we can choose $x_0$ as the companion matrix of $\chi(x)$.
Define a map
\begin{equation}\label{eq:hommap}\varphi_{\mathrm{GL}_n}:\mathrm{GL}_n\rightarrow X \ (\text{resp. } \varphi_{\mathrm{SL}_n}:\mathrm{SL}_n \rightarrow X),\ g\mapsto g^{-1}x_0 g,\end{equation}and define a right $\mathrm{GL}_n$-action (resp. $\mathrm{SL}_n$-action) on $X$ by the conjugation $x \cdot g = g^{-1} x g$. 
Hence $\varphi_{\mathrm{GL}_n}$ (resp. $\varphi_{\mathrm{SL}_n}$) is an equivariant map under the right action of $\mathrm{GL}_n$ (resp. $\mathrm{SL}_n$).

\begin{proposition}\label{prop:homogeneous}
\begin{enumerate}
    \item 
    $X$ is a homogeneous space of $\mathrm{GL}_n$ (resp. $\mathrm{SL}_n$) with respect to conjugation by $\mathrm{GL}_n$ (resp. $\mathrm{SL}_n$) and $\varphi_{\mathrm{GL}_n}$ (resp. $\varphi_{\mathrm{SL}_n}$) induces an following isomorphism
    \[
    X\cong \mathrm{T}_{x_0}\backslash \mathrm{GL_n} \ (resp.\ X\cong \mathrm{S}_{x_0}\backslash \mathrm{SL}_n),
    \]
    where $\mathrm{T}_{x_0}$ (resp. $\mathrm{S}_{x_0}$) denotes the stabilizer of $x_0$ under the $\mathrm{GL}_n$-action (resp. the $\mathrm{SL}_n$-action) on $X$.
    \item
    We have the following isomorphism for the stabilizer of $x_0$
    \[\mathrm{T}_{x_0} \cong \mathrm{R}_{K/k}(\mathbb{G}_{m,K})
    \textit{ $($resp. $\mathrm{S}_{x_0} \cong \mathrm{R}^{(1)}_{K/k}(\mathbb{G}_{m,K}))$},\]
    where we use the following notations:
    \[\left\{
    \begin{array}{l}
    \textit{$\mathrm{R}_{K/k}(\mathbb{G}_{m,K})$ is the Weil restriction of $\mathbb{G}_{m,K}$ and};\\
    \textit{$\mathrm{R}^{(1)}_{K/k}(\mathbb{G}_{m,K}):=\ker(\Nm_{K/k}: \mathrm{R}_{K/k}(\mathbb{G}_{m,K}) \to \mathbb{G}_{m,k})$ is the norm-1 torus of $K/k$}.
    \end{array}
    \right.
    \]
\end{enumerate}
\end{proposition}

\begin{proof}
    Since every square matrix over a field has a rational canonical form determined by its characteristic polynomial, for $x \in X(\bar{k})$, there exists $g \in \mathrm{GL}_n(\bar{k})$ such that $x = g^{-1} x_0 g$.
    Therefore, $X$ is a homogeneous space of $\mathrm{GL}_n$.
    Similarly, we can show that $X$ is also a homogeneous space of $\mathrm{SL}_n$ by choosing $h = \frac{g}{\sqrt[n]{\det g}} \in \mathrm{SL}_n(\bar{k})$.
    We then obtain the following identifications induced by $\varphi_{\mathrm{GL}_n}$ and $\varphi_{\mathrm{SL}_n}$ defined in (\ref{eq:hommap}),
\[X \cong \mathrm{T}_{x_0} \backslash \mathrm{GL}_n \cong \mathrm{S}_{x_0} \backslash \mathrm{SL}_n.\]

We now prove the second statement.
Since $x_0$ has a rational canonical form, there exists an invertible matrix $Q \in \mathrm{GL}_n(k)$ such that
\begin{equation}\label{eq:rational_canonical}
    x_0 = Q^{-1}
\begin{pmatrix}
	0 & \cdots & 0 & -a_0 \\
	1 & \cdots & 0 & -a_1 \\
	\vdots & \ddots & \vdots & \vdots \\
	0 & \cdots & 1 & -a_{n-1} \\
\end{pmatrix} Q.
\end{equation}
For each $1\leq i\leq n$, we define $f_i \in K = k[x]/(\chi(x))$ by
\[f_i = \sum_{k=1}^n Q_{ki} x^{k-1}.\]
Since $\{1, x, \ldots, x^{n-1}\}$ is a $k$-basis of $K$, $\{f_1, f_2, \ldots, f_n\}$ is also a basis.

For each $k$-algebra $R$, we now construct an explicit bijection 
$\mathrm{T}_{x_0} (R) \cong \mathrm{R}_{K/k}(\mathbb{G}_{m, K})(R)$.
Let $V$ be a free $R$-module $R[x]/(\chi(x))$.
With respect to the basis $\{f_1, \ldots, f_n\}$ of $V$, we have the following identification
\[\Phi:\mathrm{M}_n(R)\xrightarrow{\sim} \mathrm{End}_R(V).\]
Since the companion matrix in (\ref{eq:rational_canonical}) represents the left-multiplication by $x$ with respect to the basis $\{1, x, \ldots, x^{n-1}\}$, we have that $\Phi(x_0)$ equals to $(f(x)\mapsto x f(x)) \in \End_R(V)$ by the change of basis.

For $g \in \mathrm{M}_n(R)$, $g x_0 = x_0 g$ in $\mathrm{M}_n(R)$ if and only if $\Phi(g)(x v) = x \Phi(g)(v)$ for each $v\in V$. 
We then identify the stabilizer $\mathrm{T}_{x_0}(R) \subset \mathrm{M}_n(R)$ with the set of $R$-linear automorphisms on $V$ that commute with $x$, under the isomorphism $\Phi$.
By the $R$-linearity, we also have
\begin{equation}\label{eq:Tx0}
    \mathrm{T}_{x_0}(R) = \{ g \in \mathrm{M}_n(R) \mid g \text{ is invertible and } \Phi(g)(v_1 v_2) = v_1 \Phi(g)(v_2) \text{ for all } v_1, v_2 \in V\}.
\end{equation}

We now define the following $R$-linear map
\[\Phi_R : \mathrm{T}_{x_0}(R) \to (R[x]/(\chi(x)))^\times,\ g \mapsto \Phi(g)(1_R),\]
where $1_R$ is the identity element in $R$. 
The image of $\Phi_R$ is contained in $(R[x]/(\chi(x)))^\times$ since, for $g \in \mathrm{T}_{x_0}(R)$,
\[\Phi_R(g) \Phi_R(g^{-1})=\Phi(g)(\Phi(g)^{-1} (1_R))= 1_R.\]

We assert that $\Phi_R$ is a bijection by constructing the inverse.
For $v \in (R[x]/(\chi(x)))^\times$, let $g_v \in \mathrm{M}_n(R)$ be a matrix whose $i$-th column consists of the coefficients of $f_i v \in V$ with respect to the basis $\{f_1, \ldots, f_n\}$.
Then we have $\Phi(g_v)(f) = f v$ for all $f \in V$.
Since $v$ is invertible, $g_v$ is contained in $\mathrm{T}_{x_0}(R)$ using the identification (\ref{eq:Tx0}).
Thus, $(R[x]/(\chi(x)))^\times\to \mathrm{T}_{x_0}(R),\ v \mapsto g_v$ defines an inverse of $\Phi_R$.
One can deduce that $\Phi_R$ is functorial in $R$ and this concludes that $\mathrm{T}_{x_0} \cong \mathrm{R}_{K/k}(\mathbb{G}_{m,K})$ as $k$-schemes.

Moreover, since $\mathrm{Nm}_{K/k} : \mathrm{R}_{K/k}(\mathbb{G}_{m,K}) \to \mathbb{G}_{m,k}$ commutes with $\det : \mathrm{GL}_n \to \mathbb{G}_{m,k}$, it follows that $\mathrm{S}_{x_0} \cong \ker \mathrm{Nm}_{K/k} = \mathrm{R}^{(1)}_{K/k}(\mathbb{G}_{m, K})$.
\end{proof}

\subsection{\texorpdfstring{Orbits of  $\mathrm{GL}_n(k_v)$}{GLn(kv)}-action and \texorpdfstring{$\mathrm{SL}_n(k_v)$}{SLn(kv)}-action in \texorpdfstring{$X(k_v)$}{X(kv)}}\label{sec:geominte}

We recall useful properties of a homogeneous space $X$.
In the previous section, we have the following identifications for $X$ over $k$:
\[ 
X\cong \mathrm{T}_{x_0}\backslash \mathrm{GL}_n\cong \mathrm{S}_{x_0}\backslash \mathrm{SL}_n.
\]
By \cite[4.30]{EGM}, the homogeneous space structure is preserved under base change and thus we have
\[
X_{k_v} \cong \mathrm{T}_{x_0 , k_v}\backslash \mathrm{GL}_{n,k_v}\cong \mathrm{S}_{x_0, k_v}\backslash \mathrm{SL}_{n,k_v},
\]
for any place $v \in \Omega_k$.
The structures of $\mathrm{T}_{x_0, k_v}$ and $\mathrm{S}_{x_0,k_v}$ are described in the following lemma.
\begin{lemma}\label{lem:centralizer_str}
We have the following isomorphisms
\[
\left\{\begin{array}{l}
     
\mathrm{T}_{x_0,k_v}\cong \prod\limits_{i\in B_v(\chi)}\mathrm{R}_{K_{v,i}/ k_v}(\mathbb{G}_{m,K_{v,i}}); \\

\mathrm{S}_{x_0,k_v}\cong \ker\Big(\prod\limits_{i\in B_v(\chi)}\Nm_{K_{v,i}/k_v}:\prod\limits_{i\in B_v(\chi)}\mathrm{R}_{K_{v,i}/k_v}(\mathbb{G}_{m,K_{v,i}})\rightarrow \mathbb{G}_{m,k_v}\Big).
\end{array}\right.
\]
\end{lemma}
\begin{proof}
The describtion for $\mathrm{T}_{x_0,k_v}$ follows from \cite[3.12]{Vosk}.
For the second description, we consider the following exact sequence of algebraic groups
\begin{equation}\label{eq:ST}
    1 \to \mathrm{S}_{x_0, k_v} \to \mathrm{T}_{x_0, k_v} \to \bG_{m, k_v} \to 1.
\end{equation}
We claim that the norm map for an \'etale algebra $\prod_{i\in B_v(\chi)} K_{v,i}$ over $k_v$ is the product of norms for $K_{v,i}$ for $i \in B_v(\chi)$. 

We verify the statements for $\mathrm{S}_{x_0,k_v}$ on $A$-points for each $k_v$-algebra $A$.
We fix an element $x \in A$ and define $\phi : A \otimes_k K \to A \otimes_k K$ to be the left-multiplication by $x \otimes 1$.
Under the following composition of isomorphisms 
\[A \otimes_k K = (A \otimes_{k_v} k_v) \otimes_k K \cong A \otimes_{k_v} (k_v \otimes_k K) \cong A \otimes_{k_v} \prod_{i\in B_v(\chi)} K_{v,i} \cong \prod_{i\in B_v(\chi)} A \otimes_{k_v} K_{v,i},\]
the transfer of $\phi$ from $\prod_{i\in B_v(\chi)} A \otimes_{k_v} K_{v,i}$ to itself is also the left-multiplication by $x\otimes 1$.
We then have
\[\Nm_{A \otimes_k K/ A}(x) = \det \phi = \prod_{i\in B_v(\chi)} \det \phi_i = \prod_{i\in B_v(\chi)} \Nm_{A \otimes_{k_v} K_{v,i}/A}(x),\]
where $\phi_i: A \otimes_{k_v} K_{v,i} \to A \otimes_{k_v} K_{v,i}$ is the left-multiplication on $ A \otimes_{k_v} K_{v,i}$ by $x \otimes 1$. 
This proves the claim and concludes the lemma.
\end{proof}
We now describe the orbits of $X(k_v)$ under the $\mathrm{GL}_n(k_v)$-action and the $\mathrm{SL}_n(k_v)$-action, respectively.
\begin{proposition}\label{prop:gln_orbit}
    There is only one $\mathrm{GL}_n(k_v)$-orbit in $X(k_v)$ and thus $X(k_v)\cong \mathrm{T}_{x_0}(k_v)\backslash \mathrm{GL}_n(k_v)$.
\end{proposition}

\begin{proof}
    We consider the following exact sequence for the $\mathrm{GL}_{n,k_v}$-homogeneous space $X_{k_v}$,
    \[1 \to \mathrm{T}_{x_0,k_v} \to \mathrm{GL}_{n,k_v}\to X_{k_v} \to 1.\]
    This induces the following long exact sequence
    \[\cdots \to \mathrm{GL}_n(k_v) \to X(k_v) \to \mathrm{H}^1(k_v, \mathrm{T}_{x_0,k_v}) \to \cdots .\]
    Using Lemma \ref{lem:centralizer_str}, by \cite[3.12]{Vosk} and Hilbert's Theorem 90, we have
    \[\mathrm{H}^1(k_v, \mathrm{T}_{x_0,k_v}) \cong \prod_{i \in B_v(\chi)} \mathrm{H}^1(k_v, \mathrm{R}_{K_{v,i}/k_v}(\mathbb{G}_{m,K_{v,i}})) \cong \prod_{i \in B_v(\chi)} \mathrm{H}^1(K_{v,i}, \mathbb{G}_{m,K_{v,i}}) = 1,\]
    where $K \otimes_k k_v \cong \prod_{i \in B_v(\chi)} K_{v,i}$.
\end{proof}

\begin{proposition}\label{prop:number_orbit}
    The set of $\mathrm{SL}_n(k_v)$-orbits within $X(k_v)$ is in bijection with 
    $\mathrm{H}^1(k_v,\mathrm{S}_{x_0,k_v})$.
\end{proposition}
\begin{proof}
    We consider the following exact sequence for the $\mathrm{SL}_{n,k_v}$-homogeneous space $X_{k_v}$,
    \[ 1 \to \mathrm{S}_{x_0,k_v} \to \mathrm{SL}_{n,k_v} \to X_{k_v} \to 1.\]
    This induces the following long exact sequence
    \begin{equation}\label{eq:long_exact_sequence_for_sln}\cdots \to \mathrm{SL}_n(k_v) \to X(k_v) \to \mathrm{H}^1(k_v, \mathrm{S}_{x_0,k_v}) \to \mathrm{H}^1(k_v, \mathrm{SL}_{n,k_v}) \to \cdots.\end{equation}
    Since $\mathrm{H}^1(k_v, \mathrm{SL}_{n,k_v}) = 1$, it follows that the set of $\mathrm{SL}_n(k_v)$-orbits is in bijection with $\mathrm{H}^1(k_v, \mathrm{S}_{x_0,k_v})$.
\end{proof}

\begin{remark}\label{rmk:h1gal}
    If $K_v$ is a field extension of $k_v$, then $\mathrm{H}^1(k_v,\mathrm{S}_{x_0,k_v})$ is isomorphic to $\Gal(K_v/k_v)^{ab}$.
    Indeed, the following exact sequence is induced by (\ref{eq:ST}),
    \[K_v^\times \xrightarrow{\Nm_{K_v/k_v}} k_v^\times \to \mathrm{H}^1(k_v,\mathrm{S}_{x_0,k_v}) \to \mathrm{H}^1(k_v,\mathrm{T}_{x_0,k_v}).\]
    As shown in the proof of Proposition \ref{prop:gln_orbit}, we have $\mathrm{H}^1(k_v,\mathrm{T}_{x_0,k_v}) = 0$. 
    Thus, $\mathrm{H}^1(k_v,\mathrm{S}_{x_0,k_v})$ is isomorphic to $k_v^\times / \Nm_{K_v/k_v} K_v^\times$ and this quotient is isomorphic to $\Gal(K_v/k_v)^{ab}$ by \cite[Section 3, XI]{Ser}.
\end{remark}

\begin{definition}\label{def:slnorbits}
    We denote by $\mathcal{R}_{x_0,v}$ the set of representatives of each $\mathrm{SL}_n(k_v)$-orbits in $X(k_v)$.
    For each $x_0'\in \mathcal{R}_{x_0,v}$, we define $z_{x_0'}$ to be the element in $\mathrm{H}^1(k_v, \mathrm{S}_{x_0,k_v})$ corresponding to the orbit $x_0'\cdot \mathrm{SL}_n(k_v)$ according to Proposition \ref{prop:number_orbit}.
\end{definition}
\begin{remark}\label{rmk:orbitforsln}
    By Proposition \ref{prop:number_orbit}, we can describe $X(k_v)$ as the following disjoint union.
    \[
    X(k_v)= \bigsqcup_{x_0'\in\mathcal{R}_{x_0,v}} x_0'\cdot \mathrm{SL}_n(k_v)  \cong \bigsqcup_{x_0'\in\mathcal{R}_{x_0,v}}\mathrm{S}_{x_0'}(k_v)\backslash \mathrm{SL}_{n}(k_v),
    \]
    where $\mathrm{S}_{x_0'}$ is the centralizer of $x_0'$ under the $\mathrm{SL}_n$-conjugation.
\end{remark}

\begin{remark}
    We remark that each $\mathrm{SL}_n(k_v)$-orbit is open in $X(k_v)$.
    For each $x \in X(k_v)$, we define the equivariant map over $k_v$, $\varphi_{\mathrm{SL}_n,x} : \mathrm{SL}_n \to X$ by $g \to g^{-1}xg$.
    By the same argument in Lemma \ref{lem:centralizer_str},
    one can deduce that $\mathrm{S}_{x,k_v}\cong\ker (\prod_{i\in B_v(\chi)}\Nm_{K_{v,i}/k_v})$ and it is smooth over $k_v$.
    Thus, $\varphi_{\mathrm{SL}_n,x}$ is also smooth over $k_v$ by \cite[Corollary 4.33]{EGM} and $\varphi_{\mathrm{SL}_n,x}(k_v)$ is an open map by \cite[Proposition 3.5.73]{Poonen}.
    Therefore, the image of $\mathrm{SL}_n(k_v)$ via $\varphi_{\mathrm{SL}_n,x}(k_v)$, that is the $\mathrm{SL}_n(k_v)$-orbit of $x$, is open in $X(k_v)$.
\end{remark}

\subsection{Measures on \texorpdfstring{$X$}{X}}\label{sec:Tammeasure}
We note that $\mathrm{GL}_n$ and $\mathrm{T}_{x_0}$ are unimodular since they are connected reductive $k$-groups.
By \cite[Corollary of Theorem 2.2.2]{Wei82}, $\mathrm{GL}_n$ and $\mathrm{T}_{x_0}$ admit translation-invariant gauge forms, denoted by $\omega_{\mathrm{GL}_n}$ and $\omega_{\mathrm{T}_{x_0}}$, respectively.
Moreover, $X$ admits a $\mathrm{GL}_n$-invariant gauge form $\omega_X$ by \cite[Section 1.4]{BR}.
According to the arguments in \cite[p. 24]{Wei82}, by multiplying a constant in $k^\times$ on $\omega_X$, we may and do assume that $\omega_{\mathrm{GL}_n}$, $\omega_{\mathrm{T}_{x_0}}$, and $\omega_X$ match together algebraically in the sense of \cite[Section 2.4, p. 24]{Wei82} (see \cite[Definition 2.4]{Lee25}).
\begin{definition}\label{def:tamagawameasure}
We define the Tamagawa measure $m^X$ on $X(\mathbb{A}_k)$ with respect to the $\mathrm{GL}_n$-invariant gauge form $\omega_X$ as follows
\[m^X :=  \abs{\Delta_k}^{-\frac{1}{2} \dim X} \prod_{v \in \Omega_k} \abs{\omega_X}_v.\]
Here, the definition of the Tamagawa measure $m^X$ is indenpendent of the choice of $\omega_X$ by \cite[Lemma 6.2]{Lee25}.
\end{definition}

To make use of (\ref{eq:main_idea_intro}), we need to describe how the measure $m^X$ is computed.
For this purpose, we adopt the measure $d\mu_v$ on $X_{k_v}$ as defined in \cite[(6.3)]{Lee25}, which is more suitable for explicit computation.
\begin{equation}\label{eq:quotient_measure}
d\mu_v:=\left\{\begin{array}{l l}
     |\omega_{X_{k_v}}^{can}|_v& v\in\infty_k\\
     dg_v/dt_v&  v\in \Omega_k\setminus \infty_k.
\end{array}
\right.
\end{equation}
\begin{itemize}
    \item For $v\in\infty_k$, $\omega_{X_{k_v}}^{can}$ is a $\mathrm{GL}_{n,k_v}$-invariant form  on $X_{k_v}$ such that $\omega_{\mathrm{GL}_{n,k_v}}^{can}=\omega_{X_{k_v}}^{can}\cdot\omega_{\mathrm{T}_{x_0,k_v}}^{can}$ where $\omega_{\mathrm{GL}_{n,k_v}}^{can}$ (resp. $\omega_{\mathrm{T}_{x_0,k_v}}^{can}$) is the invariant forms on $\mathrm{GL}_{n,k_v}$ (resp. $\mathrm{T}_{x_0,k_v}$) defined in \cite[Section 9]{GG99}.
\item For $v\in\Omega_k\setminus\infty_k$, $dg_v/dt_v$ is the quotient measure on $X(k_v)\cong \mathrm{T}_{x_0}(k_v)\backslash \mathrm{GL}_n(k_v)$ where $dg_v$ (resp. $dt_v$) is the measure on $\mathrm{GL}_n(k_v)$ (resp. $\mathrm{T}_{x_0}(k_v)$) such that $vol(dg_v, \mathrm{GL}_n(\mathcal{O}_{k_v}))=1$ (resp. $vol(dt_v,T_c)=1$ for the maximal compact open subgroup $T_c$ of $\mathrm{T}_{x_0}(k_v)$).
\end{itemize}
Then the relation between $m^X$ and $\prod_{v\in \Omega_k}d\mu_v$ is given as follows
\begin{lemma}[{\cite[Lemma 6.3]{Lee25}}]\label{lem:measure_comparison_m_mu}
We have
    \[
C_{meas}:=
\frac{m^X}{\prod_{v\in \Omega_k}d\mu_v}=2^{(n-1)[k:\Q]}|\Delta_k|^{\frac{-n^2+n}{2}}  \frac{R_K h_K \sqrt{\abs{\Delta_K}}^{-1}}{R_k h_k \sqrt{\abs{\Delta_k}}^{-1}} \left(\prod_{i=2}^n \zeta_k(i)^{-1}\right)\left(\prod_{v\in\Omega_k\setminus\infty_k}\frac{|\Delta_\chi|_v^{-\frac{1}{2}}}{q_v^{S_v(\chi)}}\right),
    \]where $R_F$ is the regulator of $F$, $h_F$ is the class number of $\mathcal{O}_F$, $\zeta_k$ is the Dedekind zeta function of $k$, and $S_v(\chi)$ is the $\mathcal{O}_{k_v}$-module length between $\mathcal{O}_{K_v}$ and $\mathcal{O}_{k_v}[x]/(\chi(x))$.
\end{lemma}

\section{Brauer group and evaluation}\label{sec:Brauer}

In this section, we introduce the definition of the Brauer group and the Brauer evaluation, used in Proposition \ref{prop:main_idea_intro}.
We mainly follow \cite[Section 3]{Lee25}.
\begin{definition}[{\cite[Definition 3.1]{Lee25}}]\label{def:brauer}
    For any scheme $V$ over $k$, the Brauer group is defined by
    \[\Br V := \mathrm{H}_{\acute{e}t}^2(V, \mathbb{G}_{m}),\]
    where $\mathbb{G}_{m}$ stands for the \'etale sheaf associated with the multiplicative group over $V$ (see \cite[Proposition 6.3.19]{Poonen}).
    For an affine scheme $V=\Spec A$, we denote $\Br (\Spec A)$ by $\Br A$ to ease the notation.
\end{definition}
\begin{definition}%[{\cite[Definition 3.2]{Lee25}}]
\label{def:brauer_evaluation}
Let $V$ be a $k$-scheme and $A$ be a $k$-algebra. For the contravariant functor $\mathrm{Br}(-):= \mathrm{H}_{\acute{e}t}^2(-, \mathbb{G}_{m})$ from $\mathrm{Sch}_k$ to $\mathrm{Ab}$, we define the Brauer evaluation as follows. 
    \begin{enumerate}
        \item For any $x \in V(A)$ and $\xi \in \Br V$, we define the Brauer evaluation $\xi(x)$ to be the image of $\xi$ under the morphism $\mathrm{Br}(x):\Br V\rightarrow \Br A$.

        \item For a local field $k_v$ where $v\in \Omega_k$, let $\inv_v : \Br k_v \to \Q/\Z$ be the injection in \cite[Theorem 1.5.34]{Poonen}.
        In this case,
        for any $x \in V(k_v)$ and $\xi \in \Br V$, we define the Brauer evaluation $\xi_v(x)$ to be the image of $\xi$ under the following composite map
        \[\Br V \xrightarrow{\mathrm{Br} (x)} \Br k_v \xhookrightarrow{\inv_v} \Q / \Z\xrightarrow{x\mapsto \exp(2\pi ix)}\mathbb{C}^\times. 
        \]
    \end{enumerate}
\end{definition}

We now return to the case $X \cong \mathrm{T}_{x_0}\backslash \mathrm{GL}_n$.
For each place $v \in \Omega_k$, we shall compute the local integrals of the Brauer evaluation $\xi_v$ over $X(k_v)$ where $\xi \in\Br X$.
For this purpose, we normalize $\xi_v$ so that its value at $x_0$ is trivial.
%, which is independent of the choice in $\Br X / \Br k$
\begin{definition}\label{def:normalizedeval}
    For $\xi \in \Br X$, we define the normalized Brauer evaluation $\tilde{\xi}_v$ on $x\in X(k_v)$ to be
    \[\tilde{\xi}_v(x) := \frac{\xi_v(x)}{\xi_v(x_0)},\]
    where $v\in \Omega_k$ and $x_0$ is the point in $X(\mathcal{O}_k)$ chosen in Section \ref{subsec:homogeneous}.
\end{definition}

\begin{remark}\label{rem:normalized_eval}
    The structure morphism induces a map $\Br k \to \Br X$ via the contravariant functor $\Br(-)$.
    We denote the cokernel of this map by $\Br X /\Br k$. 
    \begin{enumerate}
        \item If $\xi$ lies in the trivial coset of $\Br X/\Br k$, then the evaluation $\xi_v$ is constant on $X(k_v)$ for all $v\in \Omega_k$ and $\prod_{v\in\Omega_k}\xi_v\equiv 1$ (see {\cite[Remark 3.3]{Lee25}}).
        By definition, one has $\tilde{\xi}_v\equiv 1$ for all $v\in\Omega_k$.
        \item The construction of $\mathrm{inv}_v$ in \cite[Theorem 1.5.34]{Poonen} implies that the injection $\inv_v:\Br k_v\rightarrow \mathbb{Q}/\mathbb{Z}$ is a group homomorphism (see \cite[p. 130]{Ser67}).
        Using this fact together with (1), we have that if $\xi$ and $\xi'$ lie in the same coset of $\Br X / \Br k$, then their normalized evaluations $\tilde{\xi}_v$ and $\tilde{\xi}_v'$ coincide.
% However, the normalized Brauer evaluation can be trivial even if $\xi$ is non-trivial in $\Br X/\Br k$.
    \end{enumerate}
    % Indeed, for a fixed embedding $k \hookrightarrow k_v$, the composition $\Br k\rightarrow \Br X \xrightarrow{\mathrm{Br}(x)} \Br k_v$ is independent of $x \in X(k_v)$.
    % It follows from Definition \ref{def:brauer_evaluation}.(2) that $\xi_v(x)$ is constant on $X(k_v)$. Then, \cite[Proposition 8.2.2]{Poonen} yields that $\sum_{v\in \Omega_k}\xi_v\equiv 0$.
    %     \item 
    % If $X\cong H\backslash G$ where $G$ is a simply connected semisimple group over $k$ and $H\subset G$ is a connected subgroup over $k$, then by Proposition 2.2 and Proposition 2.10 in \cite{CtX} one has an isomorphism \[\Br X/\Br k\cong\Pic H.\] Here, the Picard group $\Pic H$ is a finite abelian group.
    % \end{enumerate}
\end{remark}

% By Remark \ref{rem:normalized_eval}, if $\xi$ and $\xi'$ lie in the same coset of $\Br X / \Br k$, then their normalized evaluations $\tilde{\xi}_v$ and $\tilde{\xi}_v'$ coincide.
By Remark \ref{rem:normalized_eval}, one can check that the summand of the right hand side of (\ref{eq:main_idea_intro}) - which is the product of the integrals of $\xi_v$ over $v\in \Omega_k$ - is well-defined for each $\xi \in \Br X/\Br k$.\footnote{We note, however, that the integral of the evaluation $\xi_v$ is not well-defined for $\xi \in\Br X/ \Br k$.
In contrast, 
% the normalized evaluation yields a well-defined local integral for each $\xi \in \Br X/\Br k$. Hence, 
the normalization in Definition \ref{def:normalizedeval} not only fixes the value at $x_0$ to be $1$, but also ensures that the local integrals are well-defined.}
%그냥 이거 빼내고 lemma로 아래 내용을 쓰는게 더 나을지도
Since $x_0$ is in $X(k)$, we have $\prod_{v\in\Omega_k}\xi_v(x_0)=1$ for $\xi \in \Br X/\Br k$ by \cite[Proposition 8.2.2]{Poonen} and hence 
% We also have that for $\xi \in \Br X$
\begin{equation}\label{eq:normalized_evaluation}\begin{aligned}
    \prod_{v \in \Omega_k \setminus\infty_k} \int_{X(\mathcal{O}_{k_v})} \tilde{\xi}_v(x) &\ \abs{\omega_X}_v \prod_{v \in \infty_k} \int_{X(k_v, T)} \tilde{\xi}_v(x) \ \abs{\omega_X}_v \\
    &= \prod_{v \in \Omega_k \setminus\infty_k} \int_{X(\mathcal{O}_{k_v})} \xi_v(x) \ \abs{\omega_X}_v \prod_{v \in \infty_k} \int_{X(k_v, T)} \xi_v(x) \ \abs{\omega_X}_v.
\end{aligned}\end{equation}

\section{Brauer Evaluation on \texorpdfstring{$X$}{X}}
\label{sec:5}
From now on, throughout Sections \ref{sec:5}-\ref{sec:integral_computation}, we work under the assumptions:
% \begin{itemize}
%     \item $k$ is a totally real number field and,
%     \item $K$ is a totally real number field with $\deg(K/k)=n$ is a prime number.
% \end{itemize}
\begin{equation}\label{eq:condition45}
\left\{
\begin{array}{l}
     \textit{(1) $k$ and $K$ are totally real number fields};  \\
     \textit{(2) $\chi(x)$ is of prime degree $n$}.
\end{array}
\right.
\end{equation}

Our strategy to obatain the asymptotic formula for $N(X,T)$ is based on  Proposition \ref{prop:main_idea_intro}.
In this section, we analyze the Brauer evaluation $\xi_v$ using the local class field theory, while Section \ref{sec:integral_computation} is devoted to computing the local integral of $\xi_v$ at each place $v\in\Omega_k$.
Our discussion splits into two cases, depending on whether $K/k$ is Galois.

\subsection{The case that $K/k$ is not a Galois extension}

\begin{proposition}\label{prop:not_galois}
    If $K/k$ is not Galois, then $\Br X / \Br k$ is trivial.
\end{proposition}
\begin{proof}
    By the proof of \cite[Theorem 6.1]{WX} together with \cite[Proposition 2.10]{CtX}, we have
    \begin{equation}\label{eq:longexseqofBr}\Br X / \Br k \cong \Pic \mathrm{S}_{x_0} \cong\ker (\Hom(\mathrm{Gal}(L/k), \Q/\Z) \to \Hom(\mathrm{Gal}(L/K), \Q/\Z)),\end{equation}
    where $L$ denotes the Galois closure of $K/k$.
    For $\xi \in \Br X/\Br k$, let $\phi \in  \mathrm{Hom}(\mathrm{Gal}(L/k),\Q/\Z)$ be the image of $\xi$ under the composite of isomorphisms (\ref{eq:longexseqofBr}).
    As in the proof of \cite[Lemma 6.4]{Lee25}, we then have $\ker \phi=\mathrm{Gal}(L/k)$.
    It follows that $\Br X/\Br k$ is trivial.
\end{proof}

\subsection{The case that $K/k$ is a Galois extension}
\begin{proposition}\label{prop:isomorphism_Br_hom}
    If $K/k$ is Galois, then there is an isomorphism
    \begin{equation}\label{eq:Psi}
    \Psi: \Br X / \Br k  
    \xrightarrow{\cong} \Hom(\Lambda , \Q/\Z),
    \end{equation}
    where $\Lambda$ denotes $\mathrm{Gal}(K/k)$.
\end{proposition}
\begin{proof}
    By \cite[Proposition 2.10.(ii)]{CtX}, we have an isomorphism
    \[
    \Br X/\Br k\cong \Pic \mathrm{S}_{x_0}.
    \]
    We construct an isomorphism $\Pic \mathrm{S}_{x_0}\cong \mathrm{Hom}(\Lambda,\mathbb{Q}/\mathbb{Z})$ following the proof of \cite[Theorem 6.1]{WX}.
    Consider the short exact sequence of $\Lambda$-modules
    \[
    0\rightarrow \mathbb{Z}\rightarrow\mathbb{Z}[\Lambda]\rightarrow \hat{\mathrm{S}}_{x_0}\rightarrow 0
    \]where $\hat{\mathrm{S}}_{x_0}$ denotes the character group of $\mathrm{S}_{x_0}$.
    Since $\Z[\Lambda] = \Ind_1^{\Lambda} \Z$, we have
    \[\mathrm{H}^i(\Lambda, \Z[\Lambda]) \cong \mathrm{H}^i(1, \Z),\]
    which is trivial for $i=1$ and $ 2$ due to Shapiro's lemma. 
    It follows that $\mathrm{H}^1(\Lambda, \hat{\mathrm{S}}_{x_0})\cong \mathrm{H}^2(\Lambda,\mathbb{Z})$.
    
    From the exact sequence of $\Lambda$-modules
    \[
    0\rightarrow \mathbb{Z}\rightarrow \Q\rightarrow \Q/\Z\rightarrow 0,
    \]in which $\Lambda$ acts trivially on every term, we have $\mathrm{Hom}(\Lambda, \Q/\Z)=\mathrm{H}^1(\Lambda,\Q/\Z)\cong \mathrm{H}^2(\Lambda,\Q)$.
    Finally, by Theorem 2 in \cite[Section 4.3 of Chapter 2]{Vosk}, we have
    \[
    \Pic(\mathrm{S}_{x_0})\cong \mathrm{H}^1(\Lambda, \hat{\mathrm{S}}_{x_0})\cong \mathrm{H}^2(\Lambda,\Z )\cong \mathrm{Hom}(\Lambda,\Q/\Z).
    \]
\end{proof}

We denote $\Gal(K_v/k_v) \subset \Lambda$ by $\Lambda_v$.
    The arguments in the proof of Proposition \ref{prop:isomorphism_Br_hom} are also apply to $X_{k_v}$ with $\Lambda_v$, so that we have \[\Psi_v : \Br X_{k_v} / \Br k_v \xrightarrow{\cong} \Pic \mathrm{S}_{x_0, k_v} \xrightarrow{\cong} \Hom(\Lambda_v, \Q/\Z).\]
    Moreover, this yields the following commutative diagram:
\begin{equation}\label{eq:Psi_local_global}
\begin{tikzcd}
	{\Br X / \Br k} & {\Hom(\Lambda, \Q/\Z)} \\
	{\Br X_{k_v} / \Br k_v} & {\Hom(\Lambda_v, \Q/\Z)},
	\arrow["\Psi", from=1-1, to=1-2]
	\arrow[from=1-1, to=2-1]
	\arrow["res", from=1-2, to=2-2]
	\arrow["{\Psi_v}", from=2-1, to=2-2]
\end{tikzcd}\end{equation}
where $\Br X_{k_v}/\Br k_v$ denotes the cokernel of the map $\Br X_{k_v}\rightarrow \Br k_v$ induced by the structure morphism of $X_{k_v}$.
% From the description of $g$ in the lemma above, we know that the Brauer evaluation is constant on $\mathrm{SL}_n(k_v)$-orbits.

For $v\in \Omega_k$, we treat the following two cases separately: whether $K_v=K\otimes_kk_v$ is a field or not. 
By the assumption, for $v\in \infty_k$, we note that $K_v$ is not a field.
\begin{proposition}\label{prop:evaluation}
Suppose that $K/k$ is Galois and let $\Psi$ be the isomorphism (\ref{eq:Psi}).
\begin{enumerate}
    \item 
    For $v\in \Omega_k$ such that $K_v$ is a field, the normalized Brauer evaluation $\tilde{\xi}_v$ for $\xi \in \Br X$ is given by
    \[\tilde{\xi}_v(x) =  \exp 2 \pi i\Psi(\xi)( \phi_{K_v/k_v}( \det g_x )) \]
    for $x \in X(k_v)$.
    Here, $\phi_{K_v/k_v} : k_v^\times / \Nm_{K_v/k_v} K_v^\times \to \Gal(K_v/k_v) \subset \Lambda$ is the Artin reciprocity isomorphism (cf. \cite[Chapter XI]{Ser}) and $g_x \in \mathrm{GL}_n(k_v)$ such that $g_x^{-1} x_0 g_x = x$.
    \item
    For $v\in \Omega_k$ such that  $K_v$ is not a field, 
    the $\chi(x)$ splits completely into linear factors over $k_v$ and 
    the normalized Brauer evaluation $\tilde{\xi}_v$ for $\xi\in \Br X$ is given by \[\tilde{\xi}_v\equiv 1.\]
    \end{enumerate}
\end{proposition}
\begin{proof}
To ease the notation, we omit the subscript $k_v$ indicating the
base change $(-) \otimes_k k_v$ for any scheme, whenever no confusion arises.
\begin{enumerate}
    \item 
The Brauer evaluation of the homogeneous space $X$ is induced from the following commutative diagram, as shown in \cite[Proposition 2.9-2.10]{CtX},
    \begin{equation}\begin{tikzcd}[column sep = 2ex]\label{eq:diag_cup_prod}
	{X(k_v)} & \times & {\Br_* X } & \to & {\Br k_v} & \xrightarrow{\exp(2\pi i (-))\circ \inv} & {\C^\times} \\
	{\mathrm{H}^1(k_v, \mathrm{S}_{x_0})} & \times & {\mathrm{H}^1(k_v, \hat{\mathrm{S}}_{x_0})} & \xrightarrow{\cup} & {\Br k_v} & \xrightarrow{\exp(2\pi i (-))\circ \inv} & {\C^\times}.
	\arrow[from=1-1, to=2-1]
	\arrow[Rightarrow, no head, from=1-5, to=2-5]
	\arrow[Rightarrow, no head, from=1-7, to=2-7]
	\arrow["\cong" labl, from=1-3, to=2-3]
\end{tikzcd}\end{equation}
Here, $\Br_* X \subset \Br X$ denotes the subgroup of the Brauer group consisting of the elements whose evaluation vanishes at $x_0 \in X(k_v)$.
The pairing in the first row is given by the Brauer evaluation for $\Br_* X$, and that of the second row is the cup product $\cup$ in group cohomology. 
The first vertical map is from the exact sequence (\ref{eq:long_exact_sequence_for_sln}), and the second vertical map is given by the composite of the isomorphisms in \cite[Proposition 2.10]{CtX} and in Theorem 2 in \cite[Section 4.3 of Chapter 2]{Vosk}:
\[\Br_* X \cong \Br X / \Br k_v  \cong \mathrm{H}^1(k_v, \widehat{\mathrm{S}}_{x_0}).\]
The above isomorphism $\Br_* X\cong \Br X/\Br k_v$ implies that, for $\xi \in \Br X$, there is an element $\xi_0 $ in $ \Br k_v$ whose evaluation equals $\xi_v(x_0)$ and hence $\xi - \xi_0 \in \Br_* X$. 
The evaluation of $\xi-\xi_0$ then equals the normalized evaluation $\tilde{\xi}_v$.
Therefore it suffices to descirbe the evaluation of $\xi-\xi_0$.

To obtain the Brauer evaluation of $\xi-\xi_0$, we make use of the cup product in the second row of (\ref{eq:diag_cup_prod}).
     % Accordingly, we now explain how to compute this cup product.
    Since the cup product is compatible with the inflation homomorphism and both $\mathrm{S}_{x_0}$ and $\widehat{\mathrm{S}}_{x_0}$ split over $K_v$, it turns to be the following cup product in group cohomology of $\Lambda_v$-modules:
    \begin{equation}\label{eq:cup_prod}\mathrm{H}^1(\Lambda_v, \mathrm{S}_{x_0}(K_v)) \times \mathrm{H}^1(\Lambda_v, \widehat{\mathrm{S}}_{x_0}) \xrightarrow{\cup} \mathrm{H}^2(\Lambda_v, \mathrm{S}_{x_0}(K_v) \otimes_\Z \widehat{\mathrm{S}}_{x_0}) \to \mathrm{H}^2(\Lambda_v, K_v^\times)\cong \Br k_v.\end{equation}
    % Here, we abuse the notations $\mathrm{S}_{x_0, k_v}$ and $\widehat{\mathrm{S}_{x_0, k_v}}$ as $\Lambda_v$-modules of $K_v$-points.
    We claim that the following diagram commutes: 
    % allowing us to use the cup product in the second row.
    \begin{equation}\begin{tikzcd}[column sep = 2ex]\label{eq:HHSK}
	{\mathrm{H}^1(\Lambda_v,\mathrm{S}_{x_0}(K_v))} & \times & {\mathrm{H}^1(\Lambda_v,\widehat{\mathrm{S}}_{x_0}}) & \xrightarrow{\cup} & {\mathrm{H}^2(\Lambda_v,\mathrm{S}_{x_0}(K_v) \otimes_\Z \widehat{\mathrm{S}}_{x_0}) } \\
	{\mathrm{H}^0(\Lambda_v, K_v^\times)} & \times & {\mathrm{H}^2(\Lambda_v, \Z)} & \xrightarrow{\cup} & {\mathrm{H}^2(\Lambda_v, K_v^\times)},
	\arrow["\delta_1", from=2-1, to=1-1]
	\arrow["\delta_2", from=1-3, to=2-3]
	\arrow[from=1-5, to=2-5]
\end{tikzcd}\end{equation}  
    where the right vertical map is induced from the canonical pairing 
    $\mathrm{S}_{x_0}(K_v)\otimes_\mathbb{Z}\hat{\mathrm{S}}_{x_0}\rightarrow K_v^\times$ of the character group. 
    Here, $\delta_1$ and $\delta_2$ are the connecting homomorphisms arising from the first and second exact sequences below, respectively:
\begin{equation}\label{eq:exact_SZ}
0\rightarrow S_{x_0}(K_v)\rightarrow (K_v\otimes_{k_v}K_v)^\times \rightarrow K_v^\times \rightarrow 0\text{ and }0\rightarrow\Z\rightarrow\Z[\Lambda_v]\rightarrow\hat{\mathrm{S}}_{x_0}\rightarrow 0.
\end{equation}
    These two exact sequences in (\ref{eq:exact_SZ}) split. Indeed, one can  choose splitting homomorphisms $K_v^\times \to (K_v \otimes_{k_v} K_v)^\times,\ a \mapsto 1 \otimes a$ and $\Z[\Lambda_v] \to \Z,\ \sum_{\sigma\in \Lambda_v} n_\sigma \sigma \mapsto n_{id}$, respectively.
    By applying $(-) \otimes_\Z \hat{\mathrm{S}}_{x_0}$ to the first exact sequence in (\ref{eq:exact_SZ}) and $K_v^\times \otimes_\Z (-)$ to the second, we have the following diagram
    \begin{equation}\label{eq:diagram_proof}
    \begin{tikzcd}
	0 & {\mathrm{S}_{x_0}(K_v) \otimes_\Z \hat{\mathrm{S}}_{x_0}} & {(K_v \otimes_{k_v} K_v)^\times \otimes_\Z \hat{\mathrm{S}}_{x_0}}& {K_v^\times \otimes_\Z \hat{\mathrm{S}}_{x_0}} & 0 \\
	0 & {K_v^\times} & {K_v^\times \otimes_\Z \Z[\Lambda_v]} & {K_v^\times \otimes_\Z\hat{\mathrm{S}}_{x_0}} & 0,
	\arrow[from=1-1, to=1-2]
	\arrow[from=1-2, to=1-3]
	\arrow[from=1-2, to=2-2]
	\arrow[from=1-3, to=1-4]
	\arrow[dotted, from=1-3, to=2-3]
	\arrow[from=1-4, to=1-5]
	\arrow[equal, from=1-4, to=2-4]
	\arrow[from=2-1, to=2-2]
	\arrow[from=2-2, to=2-3]
	\arrow[from=2-3, to=2-4]
	\arrow[from=2-4, to=2-5]
    \end{tikzcd}\end{equation}
    where the first vertical map is given by the canonical pairing of the character group.
    Since both two functors $(-) \otimes_\Z \hat{\mathrm{S}}_{x_0}$ and $K_v^\times \otimes_\Z (-)$ preserve split exactness, two rows remain split exact.
    Hence, there is an induced map
    $(K_v \otimes_{k_v} K_v)^\times \otimes_\Z\hat{\mathrm{S}}_{x_0}\rightarrow K_v^\times\otimes_\Z \Z[\Lambda_v]$ obtained via the splitting homomorphisms, which makes the diagram commute.
    We then have the following commutative diagram
      \[\begin{tikzcd}
	{\mathrm{H}^1(\Lambda_v,K_v^\times \otimes_\mathbb{Z} \hat{\mathrm{S}}_{x_0})} & {\mathrm{H}^2(\Lambda_v, \mathrm{S}_{x_0} (K_v)\otimes_\mathbb{Z} \hat{\mathrm{S}}_{x_0})} \\
	{\mathrm{H}^1(\Lambda_v,K_v^\times \otimes_\mathbb{Z} \hat{\mathrm{S}}_{x_0})} & {\mathrm{H}^2(\Lambda_v, K_v^\times)}.
	\arrow["{\delta_3}", from=1-1, to=1-2]
	\arrow[equal, from=1-1, to=2-1]
	\arrow[from=1-2, to=2-2]
	\arrow["{\delta_4}", from=2-1, to=2-2]
\end{tikzcd}\]
where $\delta_3$ and $\delta_4$ denotes the connecting homomorphisms arising from the first and second rows in (\ref{eq:diagram_proof}), respectively. 
    For $a \in \mathrm{H}^0(\Lambda_v, K_v^\times)$ and $b \in \mathrm{H}^1(\Lambda_v, \hat{\mathrm{S}}_{x_0})$, it follows from \cite[Proposition 5, VIII]{Ser} that \[\delta_3(a \cup b) = (\delta_1 a) \cup b\text{ and }\delta_4 (a \cup b) = a \cup (\delta_2 b).\]
    Therefore the commutativity of $\delta_3$ and $\delta_4$ yields the commutativity of the diagram (\ref{eq:HHSK}).

By \cite[p. 388]{WX}, the image of $x$ in $\mathrm{H}^1(k_v,\mathrm{S}_{x_0})$ under the first vertical map of (\ref{eq:diag_cup_prod}) coincides with $\delta_1(\det g_x)$.
Let $\chi \in \Hom(\Lambda_v, \Q/\Z)$ be the image of $\xi - \xi_0$ under
    \[\Psi_v:\Br_* X  \xrightarrow{\cong} \Hom(\Lambda_v, \Q/\Z). \]
    From the construction of $\Psi_v$, the image of $\xi-\xi_0$ under the composite map
    \[
    \mathrm{Br}_*X\xrightarrow{\cong}\mathrm{H}^1(\Lambda_v,\hat{\mathrm{S}}_{x_0})\xrightarrow{\delta_2}\mathrm{H}^2(\Lambda_v,\mathbb{Z})
    \]
    coincides with $d\chi$ where $d : \Hom(\Lambda_v, \Q/\Z) = \mathrm{H}^1(\Lambda_v, \Q/\Z) \to \mathrm{H}^2(\Lambda_v, \Z)$ denotes the connecting homomorphism derived from the exact sequence \[0 \to \Z \to \Q \to \Q/\Z \to 0.\]
    By the commutativity of the diagram (\ref{eq:HHSK}), the evaluation of $\xi - \xi_0$ in (\ref{eq:diag_cup_prod}) is given by
    \[\tilde{\xi}_v(x)=(\xi-\xi_0)_v(x) = \exp 2\pi i \big(\inv(\det g_x \cup d \chi)\big) = \exp 2\pi i\Psi(\xi) (\phi_{K_v/k_v}(\det g_x)).\]
    Here, the last equality follows from \cite[Proposition 2, XI]{Ser} and the diagram (\ref{eq:Psi_local_global}).
    \item
    Since $\mathrm{Gal}(K_{v,i}/k_v)\subset \mathrm{Gal}(K/k)\cong \mathbb{Z}/n\mathbb{Z}$ and $n$ is prime, 
    $\chi(x)$ splits completely into linear factors $\chi_{v,i}(x)$ over $k_v$ where $\chi(x)=\prod_{i\in B_v(\chi)}\chi_{v,i}(x)$.
    From the proof of Proposition \ref{prop:isomorphism_Br_hom}, we have
    \[\Br X / \Br k_v \cong \Pic \mathrm{S}_{x_0} \cong \mathrm{H}^1(k_v, \widehat{\mathrm{S}}_{x_0}) .\]
    By Lemma \ref{lem:centralizer_str}, $\mathrm{S}_{x_0}$ is identified with the kernel of the multiplication map $\mathbb{G}_{m, k_v}^n \to \mathbb{G}_{m, k_v}$, and hence isomorphic to $\mathbb{G}_{m, k_v}^{n-1}$.
    Since $\mathrm{S}_{x_0}$ is split over $k_v$, the Galois action on the character group $\widehat{\mathrm{S}}_{x_0}$ is trivial.
    Therefore $\Br X / \Br k_v$ is trivial and Remark \ref{rem:normalized_eval} yields that the normalized evaluation $\tilde{\xi}_v$ is trivial.
    \end{enumerate}
\end{proof}

\begin{corollary}\label{cor:comparerationalstable}
    Suppose that $K/k$ is Galois.  
    For $v\in \Omega_k\setminus\infty_k$ such that $K_v$ is a field, the normalized Brauer evaluation $\tilde{\xi}_v$ for $\xi \in \Br X$ is constant on each $\mathrm{SL}_n(k_v)$-orbits in $X(k_v)$.
\end{corollary}
\begin{proof}
    For $x $ and $x'$ on the same $\mathrm{SL}_n(k_v)$-orbit in $X(k_v)$, we have $x = x' \cdot g$ for some $g \in \mathrm{SL}_n(k_v)$. 
    Then, by Proposition \ref{prop:evaluation}, we have $\tilde{\xi}_v(x) = \tilde{\xi}_v(x')$ for any $\xi \in \Br X$.
\end{proof}

\begin{corollary}\label{cor:nontrivial_eval}
    Suppose that $K/k$ is Galois. 
    For $v\in \Omega_k\setminus\infty_k$ such that $K_v$ is a field, if $\xi$ is a non-trivial element in $\Br X / \Br k$, then the normalized evaluation $\tilde{\xi}_v$ is non-trivial.
\end{corollary}

\begin{proof}
    Since $\Lambda_v\subset \Lambda$ is non-trivial, the hypothesis on $K/k$ directly yields that $\Lambda=\Lambda_v$.
    By Proposition \ref{prop:isomorphism_Br_hom}, $\Psi(\xi) \in \Hom(\Lambda, \Q/\Z)$ is non-trivial.
    In other words, there exists $\sigma \in \Gal(K/k)=\Gal(K_v/k_v)$ such that $\Psi(\xi)(\sigma) \neq 0$.
    Let $c \in k_v^\times$ be any lift of $\phi_{K_v/k_v}^{-1}(\sigma)\in k_v^\times / \Nm_{K_v/k_v}K_v^\times$ and we set $g = \diag(c, 1, \ldots, 1) \in \mathrm{GL}_n(k_v)$.
    Then, $\tilde{\xi}_v(x_0 \cdot g)$ is nontrivial by Proposition \ref{prop:evaluation}.
\end{proof}

\section{Compuation of local integrals}\label{sec:integral_computation}
We now compute the local integral of the Brauer evaluation $\tilde{\xi}_v$ under the assumption (\ref{eq:condition45}) in Section \ref{sec:5}.
In the case that the evaluation $\tilde{\xi}_v$ is trivial, 
it suffices to compute the volume of $X(k_v, T)$ for $v\in\infty_k$ and that of $X(\mathcal{O}_{k_v})$ for $v\in \Omega_k/\infty_k$, with respect to $d\mu_v$ defined in (\ref{eq:quotient_measure}).
Otherwise, $K_v/k_v$ is a Galois field extension by Proposition \ref{prop:not_galois} and Proposition \ref{prop:evaluation}.
In this case, we will address the two cases separately: when $K_v/k_v$ is ramified and when $K_v/k_v$ is unramified.

\subsection{The case that the evaluation $\tilde{\xi}_v$ is trivial}
% In this subsection, we treat the case that the evaluation of $\tilde{\xi}_v \in \Br X/\Br k$ is trivial.
By Proposition \ref{prop:not_galois}, Proposition \ref{prop:evaluation}, and Remark \ref{rem:normalized_eval}, the evaluation $\tilde{\xi}_v$ is trivial in the following cases:
\[\left\{\begin{array}{l}
     K/k\textit{ is not a Galois field extension};\\
     K/k\textit{ is a Galois field extension and $K_v$ is not a field};\\
     K/k\textit{ is a Galois field extension, $K_v$ is a field, and $\xi\in \Br X/\Br k$ is trivial}.
\end{array}\right.
\]
Otherwise, $K_v/k_v$ is a Galois field extension and $\xi \in \Br X/\Br k$ is non-trivial, then the evaluation $\tilde{\xi}_v$ is non-trivial by Proposition \ref{cor:nontrivial_eval}.

In the case that $v\in \infty_k$, by the assumption, $\chi(x)$ splits completely over $k_v\cong \mathbb{R}$ and hence $\tilde{\xi}_v\equiv 1$ for $v\in\infty_k$, regardless of whether $K/k$ is Galois or not.
\begin{proposition}[{\cite[Lemma A.1]{Lee25}}]\label{prop:eq_int_arch}
We have
\[ vol(d\mu_\infty,X(k_\infty,T))\sim\left( \frac{w_n\pi^{\frac{n(n+1)}{4}} }{\prod_{i=1}^n \Gamma(\frac{i}{2})} T^{\frac{n(n-1)}{2}}  \right)^{[k:\Q]}\prod_{v\in\infty_k}|\Delta_\chi|_v^{-\frac{1}{2}},\]
where $w_n$ is the volume of the unit ball in $\mathbb{R}^{\frac{n(n-1)}{2}}$.
\end{proposition}
\begin{proof}
For a fixed $v \in \infty_k$, since the polynomial $\chi(x)$ splits completely over $k_v\cong \mathbb{R}$,
there exists $g_0\in \mathrm{GL}_n(k_v)$ such that $g_0^{-1}x_0 g_0=\lambda$
where $\lambda := \diag(\lambda_1, \ldots, \lambda_n) \in X(k_v)$ and $\lambda_1, \ldots, \lambda_n \in k_v$ are the roots of $\chi(x)$.
We note that the roots $\lambda_i$ are distinct since $\lambda$ is regular.
Hence, the centralizer $\mathrm{T}_{\lambda}(k_v)$ of $\lambda$ is the set of diagonal matrices in $\mathrm{GL}_n(k_v)$. 

We can contruct the isomorphism as in the proof of \cite[Lemma A.1]{Lee25}:
\[
\phi: \mathrm{T}_\lambda(k_v)\backslash \mathrm{GL}_n(k_v)\xrightarrow{\cong} \mathrm{T}_{x_0}(k_v)\backslash \mathrm{GL}_n(k_v),\ \mathrm{T}_\lambda(k_v)g\mapsto \mathrm{T}_{x_0}(k_v) g_0g,
\]which is $\mathrm{GL}_n(k_v)$-invariant.
Then the proof directly follows from that of \cite[Lemma A.1]{Lee25}.
\end{proof}

On the other hand, for $v\in \Omega_k\setminus \infty_k$, the integral turns out to be the orbital integral for $\mathfrak{gl}_n$.
\begin{proposition}\label{prop:result_case:trivial}
For $v\in \Omega_k \setminus \infty_k$, we identify $\mathrm{M}_{n}(\mathcal{O}_{k_v})$ with $\mathfrak{gl}_{n}(\mathcal{O}_{k_v})$ so that $x_0 \in \mathfrak{gl}_{n}(\mathcal{O}_{k_v})$.
Then, we have 
    \begin{equation}\label{eq:orbital_integral}
    vol(d\mu_v, X(\mathcal{O}_{k_v}))=\int_{\mathrm{T}_{x_0}(k_v)\backslash \mathrm{GL}_n(k_v)}\mathbbm{1}_{\mathfrak{gl}_{n}(\mathcal{O}_{k_v})}(g^{-1}x_0g)\ d\mu_v.
    \end{equation}
    Here, the right hand side is called the orbital integral of $\mathfrak{gl}_n$ for $\mathbbm{1}_{\mathfrak{gl}_n(\mathcal{O}_{k_v})}$ and the characteristic polynomial $\chi(x)$ with respect to the measure $d\mu_v$ (see \cite[Section 1.3]{Yun13}). 
\end{proposition}
\begin{proof}
    By Proposition \ref{prop:gln_orbit}, we have
    \[
    X(k_v)\cong \mathrm{T}_{x_0}(k_v)\backslash \mathrm{GL}_n(k_v).
    \]
    Under this identification, we have
    $X(\mathcal{O}_{k_v})\cong \{g\in \mathrm{T}_{x_0}(k_v)\backslash \mathrm{GL}_n(k_v)\mid g^{-1}x_0g \in \mathfrak{gl}_n(\mathcal{O}_{k_v})\}$ and this yields the formula.
\end{proof}
We denote by $\mathcal{O}_{\chi,d\mu_v}(\mathbbm{1}_{\mathfrak{gl}_{n}(\mathcal{O}_{k_v})})$ the right hand side of (\ref{eq:orbital_integral}).
In general, obtaining the closed formula of the orbital integral for $\mathfrak{gl}_n$ is a challenging problem, and the difficulty increases as the rank $n$ grows.
In the case when $n=2$ and $3$, we summarize the closed formula of the product of local orbital integrals of $\mathfrak{gl}_n$ in Proposition \ref{prop:productoforbitalintegral}.

\begin{remark}\label{rmk:evalforEMS}
    By \cite[Theorem 1.5]{Yun13}, the orbital integral $\mathcal{O}_{\chi,d\mu_v}(\mathbbm{1}_{\mathfrak{gl}_n(\mathcal{O}_{k_v})})$ is formulated by a $\mathbb{Z}$-polynomial in $q_v$ (so that it is an integer).
    If $\chi(x)$ is an irreducible polynomial such that $\mathcal{O}_K=\mathcal{O}_k[x]/(\chi(x))$ (as in the assumption (2) in (\ref{eq:conditiononEMS})), $\mathcal{O}_{\chi,d\mu_v}(\mathbbm{1}_{\mathfrak{gl}_n(\mathcal{O}_{k_v})})=1$.
    % Here $S_v(\chi)$ is the $\mathcal{O}_{k_v}$-module length between $\mathcal{O}_{K_v}$ and $\mathcal{O}_{k_v}[x]/(\chi(x))$.
    % Therefore the assumption yields that $S_v(\chi)=0$ and thus (\ref{eq:EMSorbital}) holds.
\end{remark}

\subsection{The case that the evaluation $\tilde{\xi}_v$ is non-trivial and $K_v/k_v$ is ramified}
In the case that $\tilde{\xi}_v$ is non-trivial and $K_v/k_v$ is ramified, we adapt the method of the proof of \cite[Theorem 6.1]{WX}.
\begin{proposition}\label{prop:ramify_zero}
    For $v\in\Omega_k$ such that $K/k$ is Galois and $K_v/k_v$ is ramified, if
     $\xi \in \Br X/\Br k$ is non-trivial, then we have
     \[ \int_{X(\mathcal{O}_{k_v})}\tilde{\xi}_v\ d\mu_v =0.\]
\end{proposition}
\begin{proof}
 % Since $K_v$ is a ramified extension over a non-Archimedean field $k_v$, 
By \cite[(1.7) Proposition, Chapter 5]{Neu}, $K_v/k_v$ is ramified if and only if there exists an element $u\in \mathcal{O}_{k_v}^\times$ such that $u \notin \Nm_{K_v/k_v} K_v^\times$. 
Therefore the image of $u$ under the Artin reciprocity isomorphism $\phi_{K_v/k_v}$, in Proposition \ref{prop:evaluation}, is non-trivial in $\Gal(K_v/k_v)$. 
By the assumption that $K/k$ is Galois, we have that $\Gal(K/k) = \Gal(K_v/k_v)$. 
Therefore, the assumption that $\deg K_v/k_v=n$ is prime yields
\[\exp 2\pi i \Psi(\xi)(\phi_{K_v/k_v}(u)) \neq 1,\]
where $\Psi$ is defined in Proposition \ref{prop:evaluation}

We set $g := \diag(u,1,\ldots,1) \in \mathrm{GL}_{n}(\mathcal{O}_k)$. 
Since $d\mu_v$ is a $\mathrm{GL}_n(k_v)$-invariant measure and $X(\mathcal{O}_k)$ is stable under the $\mathrm{GL}_{n}(\mathcal{O}_k)$-action, by Proposition \ref{prop:evaluation}, we have
\begin{align*}
    \int_{X(\mathcal{O}_{k_v})} \tilde{\xi}_v(x)  \  d\mu_v &= \int_{X(\mathcal{O}_{k_v})} \tilde{\xi}_v(x\cdot g) \ d\mu_v \\
&= \exp 2\pi i \Psi(\xi)(\phi_{K_v/k_v}(u)) \int_{X(\mathcal{O}_{k_v}
)} \tilde{\xi}_v(x) \  d\mu_v.
\end{align*}
This directly yields that
\[\int_{X(\mathcal{O}_{k_v})} \tilde{\xi}_v \ d\mu_v=0.\]
\end{proof}

% \begin{remark}
%     holds without $n$ is prime (for $k = \Q$?)
% \end{remark}

\subsection{The case that the evaluation \texorpdfstring{$\tilde{\xi}_v$}{tilde{xi}v} is non-trivial and \texorpdfstring{$K_v/k_v$}{Kv/kv} is unramified}\label{sec:unram_cal}
In the case that $K_v/k_v$ is unramified and $\tilde{\xi}_v$ is non-trivial, we compute the integral of $\tilde{\xi}_v$ using Langlands-Shelstad fundamental lemma for Lie algebra (see \cite{Ngo}).
% By Proposition \ref{cor:comparerationalstable}, the integration of $\tilde{\xi}_v$ on $X(\mathcal{O}_{k_v})$ turns to be the weighted summation for each $\mathrm{SL}_n(k_v)$-orbit.
% In the case that $K_v/k_v$ is unramified and $\tilde{\xi}_v$ is non-trivial, we will show that the enumeration of this summation follows from the Langlands-Shelstad fundamental lemma, which is proven in \cite{Ngo}.
We begin by gathering the necessary preliminaries and stating the fundamental lemma for $\mathfrak{sl}_n$ in Section \ref{sec:endosln}.
We then provide the conclusion in Section \ref{subsec:conclusion_sln}.
Throughout this section, we fix $v\in \Omega_k\setminus \infty_k$ such that $K_v/k_v$ is an unramified field extension.

\subsubsection{The Langlands-Shelstad fundamental lemma for $\mathfrak{sl}_n$}\label{sec:endosln}
The Langlands-Shelstad fundamental lemma for Lie algebra, stated in \cite[Theorem 1]{Ngo}, represents the equation between two orbital integrals encoded by an endoscopic data.

In this section, we first introduce the endoscopic data in terms of $(\kappa, \sigma_H)$, as stated in \cite[Section 7.1]{Kot84}, rather than the endoscopic triple given in \cite[Section 3.3]{Kal} and \cite[Section 4.1.2]{Lee25}.
In particular, we present a slightly modified version of that given in \cite[Section 3]{Tom}.
Then we show that the evaluation $\tilde{\xi}_v$ on $X(\mathcal{O}_{k_v})$ for $\xi\in \Br X/\Br k$ corresponds to an endoscopic data for $\mathrm{SL}_n$, and state the fundamental lemma for $\mathfrak{sl}_n$ for the corresponding endoscopic data.

\begin{definition}[{\cite[Section 2]{Tom}}]\label{def:unram_reductive}
Let $G$ be a reductive group over $k_v$. If $G$ is a quasi-split group which splits over an unramified extension over $k_v$, then $G$ is said to be an unramified reductive group.
\end{definition}
An unramified reductive group $G$ is classified by the following root data
\[
(X^{*}(T), X_*(T), \Phi_G,\Phi_G^{\vee},\sigma_G)
\]where
\[\left\{
\begin{array}{l}
     X^*(T)\textit{ is the character group of a Cartan subgroup $T$ of $G$};  \\
     X_*(T)\textit{ is the cocharacter group of a Cartan subgroup $T$ of $G$};\\
     \Phi_G \textit{ is the set of roots};\\
     \Phi_G^\vee \textit{ is the set of coroots};\\
     \sigma_G\textit{ is
     an automorphism of finite order of $X^*(T)$ sending a set of simple roots in $\Phi_G$ to itself.}
\end{array}\right.
\]

In general, $\sigma_G$ is induced from the Frobenius automorphism of $\mathrm{Gal}(k_v^{un}/k_v)$ on the maximally split Cartan subgroup in $G$.
Here, $k_v^{un}$ is the maximal unramified extension over $k_v$, in a fixed algebraically closure $\overline{k}_v$ of $k_v$ containing $\overline{k}$.

\begin{definition}[{\cite[Section 3]{Tom}}]\label{def:endo}
    Let $G$ be an unramified reductive group over $k_v$ with the root data $(X^{*}(T), X_*(T), \Phi_G,\Phi_G^{\vee},\sigma_G)$.
    $H$ is an endoscopic group of $G$ if it is an unramified reductive group over $k_v$ whose classifying data, in Definition \ref{def:unram_reductive}, is of the form
    \[
         (X^{*}(T), X_*(T), \Phi_H,\Phi_H^{\vee},\sigma_H),
    \]
    subject to the constraints that there exists an element $\kappa \in \mathrm{Hom}(X_*(T),\mathbb{C}^\times)$ and a Weyl group element $w \in W(\Phi_G)$ such that 
    $\sigma_H=w \circ \sigma_G$, $\sigma_H(\kappa)=\kappa$, and $\Phi_H^\vee=\{\alpha\in \Phi_G^\vee\mid \kappa(\alpha)=1\}$.
\end{definition}

Now, we return to our context with $G=\mathrm{SL}_{n,k_v}$ and $T=\mathrm{S}_{x_0,k_v}$.
The following two lemmas explain that the evaluation $\tilde{\xi}_v$ assigns an endoscopic data of $\mathrm{SL}_{n,k_v}$ and present the associated endoscopic group.
\begin{lemma}\label{lem:charstofevalxi}
 % The evlauation of $\tilde{\xi}_v$ coresponds to the character $\kappa_{\tilde{\xi}_v}$ of $\mathrm{H}^1(k_v,\mathrm{S}_{x_0})$.
 The following morphism $\kappa_{\tilde{\xi}_v}$ is a character of $\mathrm{H}^1(k_v,\mathrm{S}_{x_0,k_v})$,
       \[
       \kappa_{\tilde{\xi}_v}:\mathrm{H}^1(k_v,\mathrm{S}_{x_0,k_v})\rightarrow \mathbb{C}^\times
       ,\ z_{x_0'}\mapsto \tilde{\xi}_v(x_0')
       \] 
       where $z_{x_0'}\in \mathrm{H}^1(k_v,\mathrm{S}_{x_0,k_v})$ for $x_0' \in \mathcal{R}_{x_0,v}$ is in Definition \ref{def:slnorbits}.
       % where $x_a$ is any point in $\mathrm{SL}_n(k_v)$-orbit within $X(k_v)$ which corresponds to $a \in \mathrm{H}^1(k_v,\mathrm{S}_{x_0'})$ through the identification in Prospostion \ref{prop:number_orbit}.
       % (for each $x_0' \in \mathcal{R}_{x_0,v}$ with $z_{x_0'}\in \mathrm{H}^1(k_v,\mathrm{S}_{x_0'}) $)
\end{lemma}

\begin{proof}
    % For any $z_{x_0'} \in \mathrm{H}^1(k_v,\mathrm{S}_{x_0'})$, $x_0' \in X(k_v)$ maps to $a$ through the map $X(k_v) \to \mathrm{H}^1(k_v,\mathrm{S}_{x_0'})$ in the first row of the commuatative diagram (\ref{eq:GL_SL_comm_diag}) by the proof in Prospostion \ref{prop:number_orbit}.

    By Proposition \ref{prop:number_orbit}
    % , every element in $\mathrm{H}^1(k_v,\mathrm{S}_{x_0})$ is of the form $z_{x_0'}$ for some $x_0' \in \mathcal{R}_{x_0,v}$ (cf. Definition \ref{def:slnorbits}).
    % We define a map $\kappa_{\tilde{\xi}_v} : \mathrm{H}^1(k_v,\mathrm{S}_{x_0'}) \to \C^\times$ by $z_{x_0'} \mapsto \tilde{\xi}_v(x_0')$.
    % We note that this is well-defined by 
    and Corollary \ref{cor:comparerationalstable},
    the morphism $\kappa_{\tilde{\xi}_v}$ is well-defined and independent of the choice of $\mathcal{R}_{x_0,v}$.
    % By the commutativity of (\ref{eq:GL_SL_comm_diag}), $\delta_1(\det g_{x_0'}) = z_{x_0'}$ where $g_{x_0'} \in \mathrm{GL}_n(k_v)$ such that $g_{x_0'}^{-1} x_0 g_{x_0'} = x_0'$. 
    By Proposition \ref{prop:evaluation}, we have
    \[\tilde{\xi}_v(x_0') = \exp 2 \pi i\Psi(\xi) (\phi_{K_v/k_v}(\det g_{x_0'})),\]
    where $g_{x_0'}\in \mathrm{GL}_n(k_v)$ such that $g_{x_0'}^{-1}x_0 g_{x_0'}=x_0'$.
    Since the morphisms $\Psi(\xi)$ and $\phi_{K_v/k_v}$ are group homomorphisms, $\tilde{\xi}_v$ is also a group homomorphism.
    For the connecting homomorphism $\delta$ in the exact sequence (\ref{eq:long_exact_sequence_for_sln}), we have the following commutative diagram
    \[
    \begin{tikzcd}
X(k_v) \arrow[rr, "\tilde{\xi}_v"] \arrow[rd, "\delta", two heads] &  & \mathbb{C}^\times \\
   & {\mathrm{H}^1(k_v,\mathrm{S}_{x_0,k_v}).} \arrow[ru, "\kappa_{\tilde{\xi}_v}"] &    
\end{tikzcd}
    \]Here, $\delta$ is a surjective group homomorphism, and hence $\kappa_{\tilde{\xi}_v}$ is a character of $\mathrm{H}^1(k_v,\mathrm{S}_{x_0,k_v})$.
\end{proof}

By the Tate-Nakayama isomorphism (cf. \cite[Theorem 6.5.1]{Lab}) for $\mathrm{S}_{x_0,k_v}$, the character $\kappa$ of $\mathrm{H}^1(k_v,\mathrm{S}_{x_0,k_v})$ induces an element in $\mathrm{Hom}(X_*(\mathrm{S}_{x_0,k_v}),\mathbb{C}^\times)$. 
% We first identify a character $\kappa$ of $\mathrm{H}^1(k_v,\mathrm{S}_{x_0'})$ as an element in $ \mathrm{Hom}(X_*(\mathrm{S}_{x_0'}),\mathbb{G}_m)$. %which is the connected Langlands dual group of $\mathbb{S}$.
    Since $\mathrm{S}_{x_0,k_v}\cong \mathrm{R}_{K_v/k_v}^{(1)}(\mathbb{G}_{m,K_v})$ by Proposition \ref{prop:homogeneous}.(2), we have
    \[
     \mathrm{S}_{x_0,K_v}\cong \mathbb{G}_{m,K_v}^{n-1}\textit{ and }X_*(\mathrm{S}_{x_0,k_v})\cong (\mathbb{Z}^n)_{\Sigma=0}
    \]
    where $(\mathbb{Z}^n)_{\Sigma=0}=\{(z_1,\cdots,z_n)\in\mathbb{Z}^n\mid \sum_i z_i=0\}$. 
    Here, the action of $\Lambda_v:=\mathrm{Gal}(K_v/k_v)\cong \mathbb{Z}/n\mathbb{Z}$ on $X_*(\mathrm{S}_{x_0,k_v})$, induced from that on $\mathrm{S}_{x_0,K_v}$, is described by cyclic permutations.
    On the other hand, by \cite[Proposition 6.5.2]{Lab}, we have the following isomorphism
    \[
    \mathrm{ker}(N_{\Lambda_v})/I_{\Lambda_v}X_*(\mathrm{S}_{x_0,k_v})\cong \mathrm{H}^1(k_v,\mathrm{S}_{x_0,k_v})
    \]where 
    \[\left\{
    \begin{array}{l}
        N_{\Lambda_v}: X_*(\mathrm{S}_{x_0,k_v})\rightarrow X_*(\mathrm{S}_{x_0,k_v}), \ x\mapsto \sum_{\sigma \in \Lambda_v}\sigma(x) ;\\
        I_{\Lambda_v}:=\{\tau=\sum_{\sigma\in\Lambda_v}n_\sigma \sigma\mid \sum_{\sigma\in \Lambda_v}n_\sigma =0\}.
    \end{array}
    \right.
    \]
    Since $N_{\Lambda_v}$ is a trivial morphism in $X_*(\mathrm{S}_{x_0,k_v})\cong (\mathbb{Z}^n)_{\Sigma=0}$,
    the following composition is surjective
    \begin{equation}\label{surfection}
  X_*(\mathrm{S}_{x_0,k_v})\rightarrow X_*(\mathrm{S}_{x_0,k_v})/I_{\Lambda_v}X_*(\mathrm{S}_{x_0,k_v}) \cong  \mathrm{H}^1(k_v,\mathrm{S}_{x_0,k_v}).
    \end{equation}
    Hence, the pullback of a character $\kappa$ on $\mathrm{H}^1(k_v, \mathrm{S}_{x_0,k_v})$ along this composite map defines an element in $\mathrm{Hom}(X_*(\mathrm{S}_{x_0,k_v}),\mathbb{C}^\times)$.
    We denote the induced map again by $\kappa$.

\begin{lemma}\label{lem:endo}
    % Under the above identifiaction of $\kappa_{\tilde{\xi}_v}$ as an element in $\Hom(X_*(\mathrm{S}_{x_0,k_v}),\mathbb{C}^\times)$,
    $\mathrm{S}_{x_0,k_v}$ is an endoscopic group of $\mathrm{SL}_{n,k_v}$, associated with $\kappa_{\tilde{\xi}_v}\in \mathrm{Hom}(X_*(\mathrm{S}_{x_0,k_v}),\mathbb{C}^\times)$ (see Definition \ref{def:endo}).
\end{lemma}
\begin{proof}
Let $\sigma_{\mathrm{S}_{x_0,k_v}}$ be the action on $X^*(\mathrm{S}_{x_0,k_v})$ induced from the Frobenius automorphism of $\mathrm{Gal}(K_v/k_v)$ on $\mathrm{S}_{x_0,K_v}$.
We claim that there exists an element $w$ of Weyl group for $\Phi_{\mathrm{SL}_{n,k_v}}$ such that $\sigma_{\mathrm{S}_{x_0,k_v}}=w \circ \sigma_{\mathrm{SL}_{n,k_v}}$
and $\sigma_{\mathrm{S}_{x_0,k_v}}(\kappa_{\tilde{\xi}_v})=\kappa_{\tilde{\xi}_v}$.
Since $\mathrm{SL}_{n,k_v}$ is split over $k_v$, we have $\sigma_{\mathrm{SL}_{n,k_v}}=id$. 
The Weyl group $W(\Phi_{\mathrm{SL}_{n,k_v}})$ is isomorphic to the symmetric group $S_n$ which acts on $X_*(\mathrm{S}_{x_0,k_v})\cong \mathbb{Z}^n_{\Sigma=0}$ by permuting the factors.
Thus we can find $w\in S_n$ such that $\sigma_{\mathrm{S}_{x_0,k_v}}=w\circ \sigma_{\mathrm{SL}_{n,k_v}}=w$ since $\sigma_{\mathrm{S}_{x_0,k_v}}$ acts on $X_*(\mathrm{S}_{x_0,k_v})\cong \mathbb{Z}^n_{\Sigma=0}$ by a cyclic permuation of the factors.
On the other hand, one has $id-\sigma_{\mathrm{S}_{x_0,k_v}} \in I_{\Lambda_v}$, so that $\sigma_{\mathrm{S}_{x_0,k_v}}(\kappa_{\tilde{\xi}_v})=\kappa_{\tilde{\xi}_v}$.

We now verify that 
\begin{equation}\label{eq:rootdataforH}
(X^*(\mathrm{S}_{x_0,k_v}),X_*(\mathrm{S}_{x_0,k_v}),\Phi_H,\Phi_H^\vee,\sigma_{\mathrm{S}_{x_0,k_v}})
\end{equation}
is the root data classifying an unramified reductive group $\mathrm{S}_{x_0,k_v}$, where $\Phi^\vee_H=\{\alpha \in \Phi_{\mathrm{SL}_{n,k_v}}^\vee \mid \kappa_{\tilde{\xi}_v}(\alpha)=1  \}$ and $\Phi_H$ is its dual set.
% To conclude that $(\kappa_{\tilde{\xi}_v}, \sigma_{\mathrm{S}_{x_0}})$ is an endoscopic data of $\mathrm{SL}_n$, it suffices to show that  $\sigma_{\mathrm{S}_{x_0}}$ preserves  
% $\Phi_{\mathrm{S}_{x_0}}$, which is the dual set of $\Phi_{\mathrm{S}_{x_0}}^\vee=\{\alpha \in \Phi_{\mathrm{SL}_n}^\vee \mid \kappa_{\tilde{\xi}_v}(\alpha)=1  \}$.
% Since the evaluation $\tilde{\xi}_v$ is non-trivial, $\kappa_{\tilde{\xi}_v}$ is also non-trivial by Lemma \ref{lem:charstofevalxi}.
We claim that $\Phi_H^\vee$ is empty. 
Once this is established, it follows that $\Phi_H$ is also empty. Hence, the root datum (\ref{eq:rootdataforH}) corresponds to $\mathrm{S}_{x_0,k_v}$.

We now prove the claim.
Let $\epsilon_{i,j}$ be an element $(
\{\underbrace{0,\cdots,0,1}_{i},  \underbrace{0,\cdots,0,-1}_{j-i}, \underbrace{0,\cdots,0}_{n-j})$ in $\mathbb{Z}_{\Sigma=0}\cong X_*(\mathrm{S}_{x_0,k_v})$.
Note that $\epsilon_{1,2}$ maps to a non-trivial element in $X_*(\mathrm{S}_{x_0,k_v})/I_{\Lambda_v}X_*(\mathrm{S}_{x_0,k_v})\cong \mathrm{H}^1(k_v,\mathrm{S}_{x_0,k_v})$ under the surjection (\ref{surfection}).
Since $\mathrm{H}^1(k_v,\mathrm{S}_{x_0,k_v})\cong \mathrm{Gal}(K_v/k_v)\cong \mathbb{Z}/n\mathbb{Z}$ by Remark \ref{rmk:h1gal} and
$n$ is prime, $\kappa_{\tilde{\xi}_v}$ for a non-trivial $\xi_v$ maps $\epsilon_{1,2}$ to a non-trivial $n$-th root of unity $\zeta_n$ in $\mathbb{C}^\times$ (see Corollary \ref{cor:nontrivial_eval}).
    % Now, we denote the induced morphism again by $\kappa_{\tilde{\xi}_v}$.
    For any coroot $\pm\epsilon_{i,j}\in \Phi_{\mathrm{SL}_{n,k_v}}^\vee$ where $1\leq i<j\leq n$, we then have
    \begin{equation}\label{eq:kappaoperate}
    \kappa_{\tilde{\xi}_v}(\pm\epsilon_{i,j})=
    \kappa_{\tilde{\xi}_v}(\pm\sum_{k=0}^{j-i-1}\sigma^{i+k-1}(\epsilon_{1,2}))
    =\prod_{k=0}^{j-i-1}
    \kappa_{\tilde{\xi}_v}(\pm\sigma^{i+k-1}(\epsilon_{1,2}))
    %=\kappa_{\tilde{\xi}_v}(\epsilon_{1,2})^{j-i}
    =\zeta_n^{\pm(j-i)},    \end{equation}where $\sigma \in \Lambda_v$ corresponds to the permuation $(1\ 2\ \cdots\ n)$.
    Here, the last equality in (\ref{eq:kappaoperate}) follows from the fact that $id-\sigma^{i+k-1} \in I_{\Lambda_v}$.
    Since $1\leq j-i< n$, we then have
    \[
    \Phi^\vee_H=\{\alpha \in \Phi_{\mathrm{SL}_{n,k_v}}^\vee \mid \kappa_{\tilde{\xi}_v}(\alpha)=1  \}=\emptyset.
    \]
\end{proof}

\begin{lemma}[{\cite[Theorem 1]{Ngo}}]\label{lem:fundamentallemma}
% For the endoscopic group $\mathrm{S}_{x_0,k_v}$ of $\mathrm{SL}_{n,k_v}$
Let $x_0$ be an element of $\mathfrak{sl}_n(\mathcal{O}_{k_v})$ or $\mathrm{SL}_n(\mathcal{O}_{k_v})$.
In the Lie algebra case, i.e. when $x_0 \in \mathfrak{sl}_n(\mathcal{O}_{k_v})$, we assume that $\mathrm{char}(\kappa_v) > n$.
Then we have
\begin{equation}\label{eq:SLnkappaorbitalintegral2}
\sum_{x_0'\in\mathcal{R}_{x_0,v}}
 \kappa(z_{x_0'}) \int_{\mathrm{S}_{x_0'}(k_v)\backslash \mathrm{SL}_{n}(k_v)}\mathbbm{1}_{\mathfrak{gl}_{n}(\mathcal{O}_{k_v})}(g^{-1}x_0'g)
   \ \frac{dh_v}{ds'_v}
=
|\Delta_\chi|_v^{-1/2},
\end{equation}
where $\mathrm{S}_{x_0'}$ denotes the centralizer of $x_0'$ under $\mathrm{SL}_n$-conjugation for each $x_0'\in \mathcal{R}_{x_0,v}$ and
\[
\left\{
\begin{array}{l}
\text{$dh_v$ denotes the Haar measure on $\mathrm{SL}_n(k_v)$  such that $vol(dh_v,\mathrm{SL}_n(\mathcal{O}_{k_v}))=1$};\\
\text{$ds_v'$ denotes the Haar measure on $\mathrm{S}_{x_0'}(k_v)$ such that 
$vol(ds_v',S_c')=1$} .
\end{array}
\right.
\]
Here, $S_c'$ denotes the maximal compact subgroup of $\mathrm{S}_{x_0'}(k_v)$.
\end{lemma}
\begin{proof}
By Lemma \ref{lem:endo}, $\mathrm{S}_{x_0,k_v}$ is the endoscopic group associated with $\kappa_{\tilde{\xi}_v}$.
We first consider the Lie algebra case.
Applying \cite[Theorem 1]{Ngo} with $G = \mathrm{SL}_{n,k_v}$, $H = \mathrm{S}_{x_0,k_v}$, and $\kappa = \kappa_{\tilde{\xi}_v}$,
for $x_0\in\mathfrak{sl}_n(\mathcal{O}_{k_v})$, we obtain
\begin{equation}\label{eq:bysln}
\sum_{x_0'\in\mathcal{R}_{x_0,v}}
 \kappa(z_{x_0'}) \int_{\mathrm{S}_{x_0'}(k_v)\backslash \mathrm{SL}_{n}(k_v)}\mathbbm{1}_{\mathfrak{sl}_{n}(\mathcal{O}_{k_v})}(g^{-1}x_0'g)
   \ \frac{dh_v}{ds'_v}
=
\frac{|\mathcal{D}_{\mathrm{S}_{x_0,k_v}}(x_{0,H})|_v^{1/2}}{|\mathcal{D}_{\mathrm{SL}_{n,k_v}}(x_{0})|_v^{1/2}}
\mathcal{SO}_{x_{0,H}}(\mathbbm{1}_{\mathfrak{s}_{x_0}(\mathcal{O}_{k_v})}, ds_v)
\end{equation}when $\mathrm{char}(\kappa_v)>n$.
The explanation of the terms in the right hand side is as follows:
\begin{itemize}
\item $\mathcal{D}_{\mathrm{SL}_{n,k_v}}$ and $\mathcal{D}_{\mathrm{S}_{x_0,k_v}}$
are discriminant functions of $\mathrm{SL}_{n,k_v}$ and $\mathrm{S}_{x_0,k_v}$, respectively (see \cite[Section 1.10]{Ngo}).
Since $\mathrm{S}_{x_0,k_v}$ is torus, $\mathcal{D}_{\mathrm{S}_{x_0,k_v}}$ is trivial and so 
$|\mathcal{D}_{\mathrm{S}_{x_0,k_v}}(x_{0,H}))|_v=1$.
On the other hand, by \cite[(3.3) and Example 3.6]{Gor22}, we have
\[
|\mathcal{D}_{\mathrm{SL}_{n,k_v}}(x_0)|_v=|\Delta_\chi|_v.
\]
\item
$\mathfrak{s}_{x_0}$ denotes the Lie algebra of $\mathrm{S}_{x_0}$ and 
$x_{0,H}$ is an element of $\mathfrak{s}_{x_0}(k_v)$, whose stable conjugacy class transfers to that of $x_0$ via the map given in \cite[Section 1.9]{Ngo}.
Since $\mathrm{S}_{x_0,k_v}$ is unramified, $\mathfrak{s}_{x_0}(\mathcal{O}_{k_v})$ is well-defiend.

For the maximal compact subgroup $S_c$ of $\mathrm{S}_{x_0}(k_v)$, $ds_v$ denotes the Haar measure on $\mathrm{S}_{x_0}(k_v)$ such that $vol(ds_v,S_c)=1$.
By \cite[Lemma 1.4.3 and Lemma 1.9.2]{Ngo}, $x_{0,H}\in\mathfrak{s}_{x_0}(\mathcal{O}_{k_v})$ and its centralizer in $\mathrm{S}_{x_0,k_v}$ coincides with $\mathrm{S}_{x_0,k_v}$.
This implies that the action of $\mathrm{S}_{x_0}(\bar{k}_v)$ on $\mathfrak{s}_{x_0}(k_v)$ by conjugation is trivial.
Therefore the stable orbital integral $\mathcal{SO}_{x_{0,H}}(\mathbbm{1}_{\mathfrak{s}_{x_0}(\mathcal{O}_{k_v})}, ds_v)$ (see \cite[Definition 4.1.(2)]{Lee25}) is given as follows
\[
\mathcal{SO}_{x_{0,H}}(\mathbbm{1}_{\mathfrak{s}_{x_0}(\mathcal{O}_{k_v})}, ds_v)
=\mathbbm{1}_{\mathfrak{s}_{x_0}(\mathcal{O}_{k_v})}(x_{0,H})=1.
\]
\end{itemize}
Finally, since $g^{-1}x_0'g$ lies in $\mathfrak{sl}_n(k_v)$ for $g\in\mathrm{SL}_n(k_v)$ and $x_0' \in \mathcal{R}_{x_0,v}$, we may replace the test function $\mathbbm{1}_{\mathfrak{sl}_n(\mathcal{O}_{k_v})}$ with $\mathbbm{1}_{\mathfrak{gl}_n(\mathcal{O}_{k_v})}$.

On the other hand, using the formula in \cite[Remark 4.6]{Lee25}, the same argument and computation as in the Lie algebra case apply to the group case, when $x_0\in \mathrm{SL}_n(\mathcal{O}_{k_v})$, as well. 
This yields the desired conclusion.
\end{proof}
\subsubsection{Computation}\label{subsec:conclusion_sln}
To use Langlands–Shelstad fundamental lemma, we need to compare the measure $\frac{dh_v}{ds'_v}$ in Lemma \ref{lem:fundamentallemma} and $d\mu_v$ on each $\mathrm{SL}_n(k_v)$-orbits $x_0'\cdot\mathrm{SL}_n(k_v)$ in $X(k_v)$, where $x_0'\in \mathcal{R}_{x_0,v}$.
The following lemma shows that these two measures coincide.
\begin{lemma}[{\cite[Lemma 6.6]{Lee25}}]\label{lem:measure_comparison_gln_sln}
    For $x_0'\in \mathcal{R}_{x_0,v}$, on $x_0'\cdot \mathrm{SL}_n(k_v)\cong \mathrm{S}_{x_0'}(k_v)\backslash \mathrm{SL}_n(k_v)$, we have
    \[
    \left.\mu_v\right|_{x_0'\cdot \mathrm{SL}_n(k_v)}=\frac{dh_v}{ds_v'}.
    \]
\end{lemma}
We now introduce a lemma that will be used in the proof of Proposition \ref{prop:fundlemforsln}.
\begin{lemma}\label{lem:translation}
    For $c\in \mathcal{O}_{k}$, let $X_c$ be the closed subscheme of $\mathrm{M}_{n,\mathcal{O}_k}$ representing the set of $n\times n$ matrices whose characteristic polynomial is $\chi(x-c)$.
    \begin{enumerate}
        \item 
    Then, $X_c$ is a homogeneous space of $\mathrm{GL}_n$ (resp. $\mathrm{SL}_n$), in the sence of Section \ref{subsec:homogeneous} with a point $x_c:=x_0+cI_n$, and we have the following isomorphism
    \[
    X_c\cong \mathrm{T}_{x_c}\backslash \mathrm{GL_n} \ (resp.\ X_c\cong \mathrm{S}_{x_c}\backslash \mathrm{SL}_n),
    \]
    where $\mathrm{T}_{x_c}$ (resp. $\mathrm{S}_{x_c}$) is the centralizer of $x_c$ under the $\mathrm{GL}_n$-conjugation (resp. the $\mathrm{SL}_n$-conjugation).
    \item
    There is an isomorphism $\Br X\cong \Br X_c,\ \xi\mapsto \xi_c$ where $\xi_c\in \Br X_c$ such that $\tilde{\xi}_v(x)=\tilde{\xi}_{c,v}(x+cI_n)$ and
    the following equation holds
    \[
    \int_{X(\mathcal{O}_{k_v})}\tilde{\xi}_v(x)\ d\mu_v
    =
\int_{{X_c} (\mathcal{O}_{k_v})}\tilde{\xi}_{c,v}(x)\ d\mu_v.
    \]
    % \[    
    % \int_{x_0'\cdot \mathrm{SL}_n(k_v)} |\omega_X^{can}|_v=\int_{(x_0'+cI_n)\cdot \mathrm{SL}_n(k_v)} |\omega_{X_c}^{can}|_v
    % \]
    \end{enumerate}
\end{lemma}
\begin{proof}
Since the map $\iota_c: X \rightarrow X_c,\ x \mapsto x+cI_n$ is an isomorphism which is compatible with the conjugation of $\mathrm{GL}_n$ (resp. $\mathrm{SL}_n$), we verify the first argument from the homogeneous space structure of $X$ with a point $x_0$, in Proposition \ref{prop:homogeneous}. 
This directly yields the following commutative diagram:
\begin{equation}\label{eq:diagram_translation}
    \begin{tikzcd}
X_{k_v} \arrow[d, "\sim" labl] \arrow[r, "\iota_c"]           & X_{c,k_v} \arrow[d, "\sim" labl]                \\
\mathrm{T}_{x_0,k_v}\backslash \mathrm{GL}_{n,k_v} \arrow[r, "id"] & \mathrm{T}_{x_c,k_v}\backslash \mathrm{GL}_{n,k_v}.
\end{tikzcd}
\end{equation}
Here, we note that $\mathrm{T}_{x_0,k_v}=\mathrm{T}_{x_c,k_v}$ in $\mathrm{GL}_{n,k_v}$.
    % Therefore the measure $|\omega_{X_{k_v}}^{can}|_v$ transports to the measure $|\omega_{X_{c,k_v}}^{can}|_v$ along the isomorphism $\iota_c(k_v):X(k_v)\rightarrow X_c(k_v)$, according to Section \ref{sec:measure}.

On the other hand, applying the functor $\Br=\mathrm{H}_{\acute{e}t}^2(-, \mathbb{G}_{m})$ on the isomorphism $\iota_c$ (cf. Definition \ref{def:brauer}), we have the induced isomorphism $\Br(\iota_c): \Br X_c\xrightarrow{\sim} \Br X$.
We then have the following commutative diagram for $x\in X(k_v)$,
\[
    \begin{tikzcd}[column sep=huge]
    \Br X_c & \Br k_v & &\\
    \Br X & \Br k_v & \mathbb{Q}/\mathbb{Z} &\mathbb{C}^\times.  
	\arrow["\Br(x+cI_n)",from=1-1, to=1-2]
 \arrow["\sim"labl,"\Br(\iota_c)",from=1-1, to=2-1]
 \arrow["id",from=1-2, to=2-2]
	\arrow["\Br(x)",from=2-1, to=2-2]
    \arrow["\inv",hook,from=2-2, to=2-3]
    \arrow["x \mapsto \exp(2\pi i x)",from=2-3, to=2-4]
	\end{tikzcd}
    \]
For $\xi \in \Br X$, this yields that $\xi_v(x)=\xi_{c,v}(x+cI_n)$ where $\xi_c:=\Br(\iota_c)^{-1}\xi$. 
By Definition \ref{def:normalizedeval}, we have
\[
\tilde{\xi}_v(x)=\frac{\xi_v(x)}{\xi_v(x_0)}=\frac{\xi_{c,v}(x+cI_n)}{\xi_{c,v}(x_c)}=\tilde{\xi}_{c,v}(x+cI_n)=\tilde{\xi}_{c,v}(\iota_c(k_v)(x))
\]for $x\in X(k_v)$.
By the diagram (\ref{eq:diagram_translation}), with respect to the measure $d\mu_v$ on $\mathrm{T}_{x_0,k_v}\backslash \mathrm{GL}_{n,k_v}=\mathrm{T}_{x_c,k_v}\backslash \mathrm{GL}_{n,k_v}$,
the evaluation $\tilde{\xi}_{c,v}$ on $X_c(k_v)$ is compatible with 
the evaluation $\tilde{\xi}_{v}$ on $X(k_v)$ along the isomorphism $\iota_c$. 
Thus we have that
\[
\int_{X_c (\mathcal{O}_{k_v})}\tilde{\xi}_{c,v}(y)\ d\mu_v
% =\int_{cI_n+X(\mathcal{O}_{k_v})}\tilde{\xi}_{c,v}(x)\ |\omega_{X_{c,k_v}}^{can}|_v
=\int_{X(\mathcal{O}_{k_v})}\tilde{\xi}_{c,v}(\iota_c(k_v)(x))\ d\mu_v
=\int_{X(\mathcal{O}_{k_v})}\tilde{\xi}_{v}(x)\ d\mu_v.
\]
\end{proof}

\begin{proposition}\label{prop:fundlemforsln}
Suppose that $\chi(0)=1$ or $\mathrm{char}(\kappa_v)>n$.
   For $\xi \in \Br X$ with non-trivial evaluation $\tilde{\xi}_v$, we have
   \[
   \int_{X(\mathcal{O}_{k_v})} \tilde{\xi}_v(x)\ d\mu_v
   =|\Delta_\chi|_v^{-1/2}.
   %\frac{\#\mathrm{SL}_{n, \mathcal{O}_{k_v}}(\kappa_v)\cdot q_v^{-(n^2-1)}}{\#\mathrm{S}_{x_0,\mathcal{O}_{k_v}}(\kappa_v)\cdot q_v^{-(n-1)}},
   \]
\end{proposition}
\begin{proof}
%두괄식으로 fundamental lemma 서술하고 그거의 한 side가 우리의 \xi evaluated orbital integral과 같다는 식으로 쓰는 것도 좋을듯???
If $\mathrm{char}(\kappa_v)>n$, then we can choose $c\in \mathcal{O}_k$ for $x_0$ such that $x_c:=x_0+cI_n\in \mathfrak{sl}_n(\mathcal{O}_{k_v})$. Accordingly, we can take $\mathcal{R}_{x_c,v}$ to be $\{x_0'+cI_n\mid x_0'\in\mathcal{R}_{x_0,v}\}\subset \mathfrak{sl}_n(\mathcal{O}_{k_v})$  as well.
Since the discriminant $\Delta_\chi$ is invaraint under the translation on $\chi(x)$ and $\mathrm{S}_{x_0,k_v}=\mathrm{S}_{x_c,k_v}$, we may and do assume that $x_0\in \mathfrak{sl}_n(\mathcal{O}_{k_v})$ with $\mathcal{R}_{x_0,v}\subset \mathfrak{sl}_n(\mathcal{O}_{k_v})$, by Lemma \ref{lem:translation}, in the case that $\mathrm{char}(\kappa_v)>n$.

% Now we return to the case when $\chi(0)=1$ or $\mathrm{char}(\kappa_v)>n$.
For the character $\kappa_{\tilde{\xi}_v}$ of $\mathrm{H}^1(k_v,\mathrm{S}_{x_0,k_v})$ associated with $\tilde{\xi}_v$
in Lemma \ref{lem:charstofevalxi}, 
under the identification $x_0'\cdot \mathrm{SL}_n(k_v)  \cong \mathrm{S}_{x_0'}(k_v)\backslash \mathrm{SL}_{n}(k_v)$ in Remark \ref{rmk:orbitforsln} for each $x_0'\in \mathcal{R}_{x_0}$,
we have
\begin{equation}
    \int_{X(\mathcal{O}_{k_v})} \tilde{\xi}_v(x)\ d\mu_v=
\sum_{x_0'\in\mathcal{R}_{x_0,v}}\kappa_{\tilde{\xi}_v}(z_{x_0'})\cdot \int_{\left(\mathrm{S}_{x_0'}(k_v)\backslash \mathrm{SL}_{n}(k_v)\right) }\mathbbm{1}_{\mathfrak{gl}_n(\mathcal{O}_{k_v})}(g^{-1}x_0'g)
   \ d\mu_v.
\end{equation}Here, $z_{x_0'}\in \mathrm{H}^1(k_v,\mathrm{S}_{x_0,k_v})$ corresponds to the orbit $x_0'\cdot \mathrm{SL}_n(k_v)$ in the sense of Definition \ref{def:slnorbits}. 
Then, by Lemmas \ref{lem:fundamentallemma}-\ref{lem:measure_comparison_gln_sln}, we have the desired result.
\end{proof}

\section{Main Result}\label{sec:main}
We provide our main theorem on the asymptotic formula for $N(X, T)$ based on the results in Sections \ref{sec:5}-\ref{sec:integral_computation}.
We note that the product $\prod_{v \in \Omega_k \setminus\infty_k } \mathcal{O}_{\chi,d\mu_v}(\mathbbm{1}_{\mathfrak{gl}_{n}(\mathcal{O}_{k_v})})$ which appears in the following theorem is a finite product by Remark \ref{rmk:SOfinite}. 

% We summarize the notations for the main theorem. Let $k$ be a number field, $\chi(x) \in \mathcal{O}_k[x]$ be an irreducible monic polynomial of degree $n$, and $K = k[x]/(\chi(x))$ be an extension of $k$.
% Let $X$ be an $\mathcal{O}_k$-scheme representing the set of $n\times n$ matrices whose characteristic polynomial is $\chi(x)$.
% We define \[N(X, T)=\#\{x\in X(\mathcal{O}_k)\mid \norm{x}\leq T\},\]
% for $T>0$, where the norm $\norm{\cdot}$ is defined in (\ref{def:norm}). 
\begin{theorem}\label{thm:main_thm}
Let $\chi(x) \in \mathcal{O}_k[x]$ be an irreducible monic polynomial of degree $n$.
Let $X$ be an $\mathcal{O}_k$-scheme representing the set of $n\times n$ matrices whose characteristic polynomial is $\chi(x)$.
We define \[N(X, T)=\#\{x\in X(\mathcal{O}_k)\mid \norm{x}\leq T\},\]for $T>0$, where the norm $\norm{\cdot}$ is defined in (\ref{def:norm}).

Suppose that $k$ and $K = k[x]/(\chi(x))$ are totally real, and
if $k \neq \mathbb{Q}$, we further assume that $n$ is a prime number.  We then have the following asymptotic formulas.
\begin{enumerate}
        \item If $K/k$ is not Galois or ramified Galois, then
        \[ N(X, T) \sim C_T \prod_{v \in \Omega_k \setminus\infty_k } \frac{\mathcal{O}_{\chi,d\mu_v}(\mathbbm{1}_{\mathfrak{gl}_{n}(\mathcal{O}_{k_v})})}{q_v^{S_v(\chi)}}.\]
    
        \item If $K/k$ is unramified Galois, and $\chi(0)=1$ or $K_v$ splits over $k_v$ for all $p$-adic places with $p \leq n$, then
        \[N(X, T) \sim C_T \left(  \prod_{v \in \Omega_k \setminus\infty_k }\frac{\mathcal{O}_{\chi,d\mu_v}(\mathbbm{1}_{\mathfrak{gl}_{n}(\mathcal{O}_{k_v})})}{q_v^{S_v(\chi)}} +n-1\right).\]        
\end{enumerate}
    Here, $\mathcal{O}_{\chi,d\mu_v}(\mathbbm{1}_{\mathfrak{gl}_n(\mathcal{O}_{k_v})})$ denotes the orbital integral of $\mathfrak{gl}_n$ for $\mathbbm{1}_{\mathfrak{gl}_n(\mathcal{O}_{k_v})}$ and the characteristic polynomial $\chi(x)$ with respect to the measure $d\mu_v$ defined in (\ref{eq:quotient_measure}), $S_v(\chi)$ is the $\mathcal{O}_{k_v}$-module length between $\mathcal{O}_{K_v}$ and $\mathcal{O}_{k_v}[x]/(\chi(x))$, 
    \[C_T:= |\Delta_k|^{\frac{-n^2+n}{2}} \frac{R_K h_K \sqrt{\abs{\Delta_K}}^{-1}}{R_k h_k \sqrt{\abs{\Delta_k}}^{-1}} \left(\prod_{i=2}^n \zeta_k(i)^{-1}\right)\left( \frac{2^{n-1}w_n\pi^{\frac{n(n+1)}{4}} }{\prod_{i=1}^n \Gamma(\frac{i}{2})} T^{\frac{n(n-1)}{2}}  \right)^{[k:\Q]},\]
    and we use the following notations:
    \begin{itemize}
        \item $R_F$ is the regulator of $F$, and $h_F$ is the class number of $\mathcal{O}_F$ for $F=k$ or $K$. 
        \item $w_n$ is the volume of the unit ball in $\mathbb{R}^{\frac{n(n-1)}{2}}$, and $\zeta_k$ is the Dedekind zeta function of $k$.
    \end{itemize}
\end{theorem}

\begin{proof}  
    Recall the formula in Proposition \ref{prop:main_idea_intro}.
    Applying the equation (\ref{eq:normalized_evaluation}) and Lemma \ref{lem:measure_comparison_m_mu}, we have
    \begin{equation}\label{eq:conclusioneq}
    N(X, T) \sim C_{meas} \sum_{\xi \in \Br X / \Br k}\int_{X(k_\infty, T)} \tilde{\xi}_\infty \ d\mu_\infty \prod_{v \in \Omega_k \setminus\infty_k} \int_{X(\mathcal{O}_{k_v})} \tilde{\xi}_v \ d\mu_v. \end{equation}
    Since $\chi(x)$ splits completely over $k_v$ for all $v\in\infty_k$, the evaluation $\tilde{\xi}_\infty$ is trivial by Proposition \ref{prop:evaluation}.(2). Hence, Proposition \ref{prop:eq_int_arch} yields that
    \[
   N(X,T)\sim C_T\sum_{\xi\in\Br X/\Br k}\prod_{v\in \Omega_k\setminus\infty_k}\frac{1}{q_v^{S_v(\gamma)}} \int_{X(\mathcal{O}_{k_v})}\tilde{\xi}_v(x)d\mu_v.
   \]Here, we note that $\prod_{v\in\Omega_k}|\Delta_\chi|_v^{-1/2}=1$ by the product formula for $\Delta_\chi \in k$.

We now compute the integration of $\tilde{\xi}_v(x)$ on $X(\mathcal{O}_{k_v})$ with respect to $\abs{\omega_{X_{k_v}}^{can}}_v$ for each $v\in \Omega_k\setminus \infty_k$. 
% For $k = \mathbb{Q}$, we note that the only case (1) is possible since every extension of $\mathbb{Q}$ is ramified.
\begin{enumerate}
    \item
    In the case that $K/k$ is not Galois, $\Br X/\Br k$ is trivial by Proposition \ref{prop:not_galois}.
    In the case that $K/k$ is Galois and ramified, the summand in (\ref{eq:conclusioneq}) corresponding to a non-trivial element $\xi \in \Br X / \Br k$ vanishes by Corollary \ref{cor:nontrivial_eval} and Proposition \ref{prop:ramify_zero}.
    Moreover, when $k=\mathbb{Q}$, the same conclusion holds, even without assuming that $n$ is prime, as shown in the proof of \cite[Theorem 6.1]{WX}.
    % Therefore in case (1), it suffices to enumerate the volume of  $X(\mathcal{O}_{k_v})$ with respect to $\abs{\omega_{X_{k_v}}^{can}}_v$ by Remark \ref{rmk:evaltrivial}.
    % the formula (\ref{eq:conclusioneq}) turns into \[N(X, T) \sim \abs{\Delta_k}^{\frac{-n^2+n}{2}} \prod_{v \in \Omega_k \setminus\infty_k} \int_{X(\mathcal{O}_{k_v})} \ \abs{\omega_X}_v \prod_{v \in \infty_k} \int_{X(k_v, T)} \ \abs{\omega_X}_v.\]
    % \[\frac{\# \mathrm{T}_{x_0,v}(\kappa_v)}{q_v^n} = \frac{1}{q_v^n} \prod_{i \in B_v(\chi)} \# \mathrm{R}_{K_{v,i} / k_v}(\mathbb{G}_{m,K_{v,i}})(\kappa_v) = \frac{1}{q_v^n} \prod_{w \mid v} (\abs{\kappa_w} - 1 ) = \prod_{w \mid v} ( 1- \frac{1}{\abs{\kappa_w}} )\]
    Therefore, we have
    \[ N(X, T) \sim C_T \prod_{v \in \Omega_k \setminus\infty_k } \frac{\mathcal{O}_{\chi,d\mu_v}(\mathbbm{1}_{\mathfrak{gl}_n(\mathcal{O}_{k_v})})}{q_v^{S_v(\chi)}}. \]
% We note that when $k=\Q$, every extension of $\Q$ is ramified, so the above results also cover this case, without the assumption that $\deg \chi$ is prime.
    \item 
    % For unramified extension $K/k$, there can be $\xi \in \Br X / \Br k$ such that $\tilde{\xi}_v$ is non-trivial.
    In the case that $\xi \in \Br X / \Br k$ is trivial, the computation of the integral of $\tilde{\xi}_v$ over $X(\mathcal{O}_{k_v})$ is exactly the same as that in (1).
    For a non-trivial element $\xi \in \Br X / \Br k$, we address the following two cases separately; $K_v$ splits over $k_v$ and $K_v$ is unramified over $k_v$.

    \begin{itemize}

        \item 
     If $K_v$ splits over $k_v$, then $\chi(x)$ splits completely into linear factors and the evaluation $\tilde{\xi}_v$ is trivial by Proposition \ref{prop:evaluation}.(1).
     We then have
    \[\frac{1}{q_v^{S_v(\chi)}}\int_{X(\mathcal{O}_{k_v})} \tilde{\xi}_v\ d\mu_v =   \frac{\mathcal{O}_{\chi,d\mu_v}(\mathbbm{1}_{\mathfrak{gl}_n(\mathcal{O}_{k_v})})}{q_v^{S_v(\chi)}}.\]
 By \cite[Corollary 4.10]{Yun13}, we have 
\[
\mathcal{O}_{\chi,d\mu_v}(\mathbbm{1}_{\mathfrak{gl}_n(\mathcal{O}_{k_v})})=q_v^{S_v(\chi)-\sum\limits_{i\in B_v(\chi)}S_v(\chi_{v,i})}\prod_{i\in B_v(\chi)}\mathcal{O}_{\chi_{v,i},d\mu_v}(\mathbbm{1}_{\mathfrak{gl}_n(\mathcal{O}_{k_v})}).
\]
Here, the orbital integral $\mathcal{O}_{\chi_{v,i},d\mu_v}(\mathbbm{1}_{\mathfrak{gl}_n(\mathcal{O}_{k_v})})$ for the linear characteristic polynomial $\chi_{v,i}$ equals to 1, and
the $\mathcal{O}_{k_v}$-module length $S_v(\chi_{v,i})$ between $\mathcal{O}_{K_{v,i}}=\mathcal{O}_{k_v}$ and $\mathcal{O}_{k_v}[x]/(\chi_{v,i}(x))$ equals to 0.
We then have 
\[
\mathcal{O}_{\chi_{v,i},d\mu_v}(\mathbbm{1}_{\mathfrak{gl}_n(\mathcal{O}_{k_v})})=q_v^{S_v(\chi)}.
\]
\item
If $K_v/k_v$ is an unramified Galois extension, then the evaluation $\tilde{\xi}_v$ is non-trivial by Corollary \ref{cor:nontrivial_eval}, and $\chi(0)=1$ or $\mathrm{char}(\kappa_v)>n$.
Hence Proposition \ref{prop:fundlemforsln} directly yields that
   \[ \int_{X(\mathcal{O}_{k_v})} \tilde{\xi}_v\ d\mu_v
   =|\Delta_\chi|_v^{-1/2}=q_v^{S_v(\chi)}.
   \]
  Here, $\frac{\ord_v(\Delta_\chi )}{2}=S_v(\chi)+\frac{\ord_v(\Delta_{K_v/k_v})}{2}=S_v
   (\chi)$ by \cite[Proposition 2.5]{CKL} and \cite[Theorem 2.6 in Chapter III]{Neu}.
   \end{itemize}
   To sum up, for a non-trivial $\xi \in \Br X / \Br k$, we have 
   \begin{align*}
       \prod_{v \in \Omega_k \setminus\infty_k} \frac{1}{q_v^{S_v(\chi)}}\int_{X(\mathcal{O}_{k_v})}\tilde{\xi}_v \ d\mu_v &=
       1.
   \end{align*}
   % \[\prod_{v \in \Omega_k \setminus\infty_k} \frac{1}{q_v^{S_v(\chi)}}\int_{X(\mathcal{O}_{k_v})} \ \abs{\omega_X^{can}}_v = \left( 2^{(n-1)[k:\Q]} \frac{R_K h_K \sqrt{\abs{\Delta_k}}^{-1}}{R_k h_k \sqrt{\abs{\Delta_k}}^{-1}} \prod_{i=2}^n \zeta_k(i)^{-1} \right) \prod_{\substack{v \in \Omega_k \setminus\infty_k \\ v: splits}} \frac{1}{q_v^{S_v(\chi)}} \prod_{\substack{v \in \Omega_k \setminus\infty_k \\ v: unram}} 1. \]
    Proposition \ref{prop:evaluation} yields that $\#(\Br X / \Br k)=n$ and hence
    we have $(n-1)$ non-trivial elements in $\Br X /\Br k$
   \[N(X, T) \sim C_T \left(  \prod_{v \in \Omega_k \setminus\infty_k }\frac{\mathcal{O}_{\chi,d\mu_v}(\mathbbm{1}_{\mathfrak{gl}_n(\mathcal{O}_{k_v})})}{q_v^{S_v(\chi)}} +n-1\right).\]
\end{enumerate}
\end{proof}

\begin{remark}\label{rmk:our_result_EMS}
    Suppose that $k=\Q$ and $\chi(x)$ splits completely over $\R$ (assumptions (\ref{eq:conditiononEMS}) without the assumption (2) that $\Z[x]/(\chi(x)) = \mathcal{O}_K$).
    Since every extension of $\Q$ is ramified, the asymptotic formula for $N(X,T)$ follows from Theorem \ref{thm:main_thm}.(1). 
    We then have
     \[
    N(X,T)\sim \left(\frac{2^{n-1}h_K R_K w_n}{\sqrt{\Delta_K} \prod_{k=2}^n\Lambda(k/2)}\prod_{v\in \Omega_k\setminus \infty_k}\frac{\mathcal{O}_{\chi,d\mu_v}(\mathbbm{1}_{\mathfrak{gl}_n(\mathcal{O}_{k_v})})}{q_v^{S_v(\chi)}}\right)T^{\frac{n(n-1)}{2}},
    \]
    where $\Lambda(k/2)=\pi^{-k/2}\Gamma(k/2)\zeta(k)$.

    We verify that the above formula coincides with  (\ref{prop:EMS_intro}) when $\Z[x]/(\chi(x)) = \mathcal{O}_K$.
    Indeed, since $\Z[x]/(\chi(x)) = \mathcal{O}_K$, equivalently $S_v(\chi) = 0$ for all $v \in \Omega_k \setminus \infty_k$, we have $\mathcal{O}_{\chi,d\mu_v}(\mathbbm{1}_{\mathfrak{gl}_n(\mathcal{O}_{k_v})}) = 1$ by Remark \ref{rmk:evalforEMS} and $\Delta_K=\Delta_\chi$.
    Therefore we have
    \[
    N(X,T)\sim \frac{2^{n-1}h_{K} R_{K} w_n}{\sqrt{\Delta_{\chi}}\cdot \prod_{k=2}^n\Lambda(k/2)}T^{\frac{n(n-1)}{2}}.
    \]
\end{remark}
\vspace{1em}

\appendix
\section{Orbital integrals for $\mathfrak{gl}_n$ when $n=2$ and $3$}\label{app:orbital_integral}

In the case that $n=2$ and $3$, we refer to the results in \cite{CKL}, for the closed formula of $\prod_{v \in \Omega_k \setminus\infty_k } \mathcal{O}_{\chi,d\mu_v}(\mathbbm{1}_{\mathfrak{gl}_n(\mathcal{O}_{k_v})})$.

\begin{proposition}\label{prop:productoforbitalintegral}
\begin{enumerate}
    \item 
    For $n = 2$, the product of orbital integrals $\mathcal{O}_{\chi,d\mu_v}(\mathbbm{1}_{\mathfrak{gl}_n(\mathcal{O}_{k_v})})$ is given as follows
    \[\prod_{v \in \Omega_k \setminus\infty_k } \mathcal{O}_{\chi,d\mu_v}(\mathbbm{1}_{\mathfrak{gl}_n(\mathcal{O}_{k_v})}) = \prod_{v \in \Omega_k \setminus\infty_k }\left( 1+\frac{\#\mathrm{T}_{x_0, \mathcal{O}_{k_v}}(\kappa_v)}{q_v-1} \frac{q_v^{S_v(\chi)}-1}{q_v-1} \right).\]
\item
For $n=3$, the product of orbital integrals $\mathcal{O}_{\chi,d\mu_v}(\mathbbm{1}_{\mathfrak{gl}_n(\mathcal{O}_{k_v})})$ is given as follows
\[
\prod_{v \in \Omega_k \setminus\infty_k } \mathcal{O}_{\chi,d\mu_v}(\mathbbm{1}_{\mathfrak{gl}_n(\mathcal{O}_{k_v})}) =
\prod_{v \in \Omega_k \setminus\infty_k }q_v^{\rho_v(\chi)}\left(  1+\frac{\#\mathrm{T}_{x_0, \mathcal{O}_{k_v}}(\kappa_v)}{(q_v-1)^2}
 \Phi_v(\chi) \right).
\]
Here,
\[
\left\{
\begin{array}{l}
\#\mathrm{T}_{x_0,\mathcal{O}_{k_v}}(\kappa_v)=\prod\limits_{i\in B_v(\chi)}q_v^{[K_{v,i}:k_v]-[\kappa_{K_{v,i}}:\kappa_v]}(q_v^{[\kappa_{K_{v,i}}:\kappa_v]}-1),\\
 \rho_v(\chi):= \sum\limits_{{\{i,j\}\subset B_v(\chi),i\neq j}}\ord_v (\mathrm{Res}(\chi_{v,i},\chi_{v,j}))= 
      \left\{
      \begin{array}{l l}
        0   & \textit{if }|B_v(\chi)|=1;\\
        S_v(\chi)-\delta_v  & \textit{if }|B_v(\chi)|=2;\\
        S_v(\chi)&\textit{if }|B_v(\chi)|=3,
      \end{array}\right.\\
\Phi_v(\chi)=\left\{
\begin{array}{l l}
     \frac{q_v^{\delta_v}-1}{q_v-1}&\textit{if $K_v$ splits over $k_v$};\\
     \frac{q_v^{\delta_v}-1}{q_v-1}-\frac{3(q_v^{\delta_v -d_v}-1)}{q_v^2-1}&\textit{if $K_v$ is an unramified field extension of $k_v$};  \\
     \frac{q_v^{\delta_v}-1}{q_v-1}-\frac{3(q_v^{\delta_v -d_v}-1)}{q_v^2-1}+\frac{(1+\delta_v-3d_v)q_v^{\delta_v-d_v}-1}{q_v(q_v+1)} &\textit{if $K_v$ is a ramified field extension of $k_v$},
\end{array}
\right.
\end{array}
\right.
\]
where 
$\left\{
\begin{array}{l}
       \textit{$B_v(\chi)$ is an index set in bijection with irreducible factors $\chi_{v,i}$ of $\chi$ over $k_v$};\\
       % \textit{$n_{v,i}$ denotes the degree of $\chi_{v,i}$};\\
       \textit{$\mathrm{Res}(\chi_{v,i},\chi_{v,j})$ denotes the resultant of two polynomials $\chi_{v,i}$ and $\chi_{v,j}$};\\
       
      % \delta_v:=\max\limits_{i \in B_v(\chi)}\{S_v(\chi_{v,i})\}\textit{ for the $\mathcal{O}_{k_v}$-module length $S_v(\chi_{v,i})$ between $\mathcal{O}_{K_{v,i}}$ and $\mathcal{O}_{k_v}[x]/(\chi_{v,i}(x))$};\\

        \delta_v:=\max\limits_{i \in B_v(\chi)}\{S_v(\chi_{v,i})\};\\
        S_v(\chi_{v,i}) \textit{ is the $\mathcal{O}_{k_v}$-module length between $\mathcal{O}_{K_{v,i}}$ and $\mathcal{O}_{k_v}[x]/(\chi_{v,i}(x))$};\\

    d_v:=\lfloor \frac{\delta_v}{3} \rfloor \textit{ for the floor function  $\lfloor \cdot \rfloor$}.
\end{array}
\right.
$
\end{enumerate}
\end{proposition}

\begin{proof}
\begin{enumerate}
\item
   By  \cite[Remark 5.7]{CKL}, we have
    \[
            \mathcal{O}_{\chi,d\mu_v}(\mathbbm{1}_{\mathfrak{gl}_n(\mathcal{O}_{k_v})}) =
    \left\{
    \begin{array}{l l}
         q_v^{S_v(\chi)}&  \textit{if $K_v$ splits over $ k_v$};\\
         1+(q_v+1) \frac{q_v^{S_v(\chi)}-1}{q_v-1}&  \textit{if $K_v$ is an unramified field extension over $k_v$};\\
         \frac{q_v^{S_v(\chi)+1}-1}{q_v-1}& \textit{if $K_v$ is a ramified field extension over $k_v$}.
    \end{array}
    \right.
        \]
    In order to write the above in a uniform way, we describe  $\mathrm{T}_{x_0, \mathcal{O}_{k_v}}(\kappa_v)$ explicitly as follows.
    \[\mathrm{T}_{x_0, \mathcal{O}_{k_v}}(\kappa_v)\cong
    \left\{
    \begin{array}{l l}
         (\kappa_v\times \kappa_v)^\times&  \textit{if $K_v$ splits over $ k_v$};\\
         \kappa_{K_v}^\times &  \textit{if $K_v$ is an unramified field extension over $k_v$};\\
         (\kappa_v[x]/(x^2))^\times & \textit{if $K_v$ is a ramified field extension over $k_v$},
    \end{array}
    \right.
    \]
    where in the second case, $\kappa_{K_v}$ is a quadratic field extension over $\kappa_v$.
Plugging it into the above formula for $\mathcal{O}_{\chi,d\mu_v}(\mathbbm{1}_{\mathfrak{gl}_n(\mathcal{O}_{k_v})})$, we obtain the desired formula.
   % Finally $S(\gamma)=0$ for all but finitely many $v\in |\mfo|$, since $S(\gamma)=\frac{1}{2}\ord_v(\Delta_\chi/\Delta_{K/k})$ by Remark \ref{rmk:onSerreinv}.(1) and the fact that $\Delta_\phi/\Delta_{K/k}\in \mfo$.
\item   By  \cite[Remark 6.7]{CKL}, we have
\begin{enumerate}
\item
 If $\chi(x)$ is an irreducible monic polynomial $\mathcal{O}_{k_v}$, then we have  
\begin{equation*}{\small
\mathcal{O}_{\chi,d\mu_v}(\mathbbm{1}_{\mathfrak{gl}_n(\mathcal{O}_{k_v})})=\left\{\begin{array}{l l}
(q_v^2+q_v+1)\sum\limits_{i=1}^{d_v}(q_v^{3i-2}+2q_v^{3i-3}+3q_v^{2i-2}\frac{q_v^{i-1}-1}{q_v-1})+1 & \textit{if $K_v/k_v$ is unramified and $S_v(\chi)=3d_v$};\\
\sum\limits_{i=1}^{d_v}(q_v^{3i}+2q_v^{3i-1}+q_v^{2i-1}+3q_v^{2i}\frac{q_v^{i-1}-1}{q_v-1})+1& \textit{if $K_v/k_v$ is ramified and $S_v(\chi)=3d_v$};\\
\sum\limits_{i=1}^{d_v}(q_v^{3i+1}+2q_v^{3i}+2q_v^{2i}+3q_v^{2i+1} \frac{q_v^{i-1}-1}{q_v-1})+q_v+1& \textit{if $K_v/k_v$ is ramified and  $S_v(\chi)=3d_v+1$}.\\
\end{array}\right.}
\end{equation*}
Here, we note that $S_v(\chi)$ cannot be of the form $3d_v+2$ (cf. \cite[Proposition 6.2]{CKL}).

\item If $\chi(x)=\chi_{v,1}(x)\chi_{v,2}(x)$ over $\mathcal{O}_{k_v}$ where $\chi_{v,1}(x)\in\mathcal{O}_{k_v}[x]$ is an irreducible monic quadratic polynomial, we have
\[
\mathcal{O}_{\chi,d\mu_v}(\mathbbm{1}_{\mathfrak{gl}_n(\mathcal{O}_{k_v})})=\left\{\begin{array}{l l}
q_v^{S_v(\chi)-S_v(\chi_{v,1})} (1+(q_v+1)\frac{q_v^{S_v(\chi_{v,1})}-1}{q_v-1})& \textit{if $K_{v,1}/k_v$ is unramified};\\
q_v^{S_v(\chi)-S_v(\chi_{v,1})} \frac{q_v^{S_v(\chi_{v,1})+1}-1}{q_v-1} &  \textit{if $K_{v,1}/k_v$ is ramified}.
\end{array}\right.
\]

\item If $\chi(x)$ splits into linear terms completely over $\mathcal{O}_{k_v}$, we have
\[
\mathcal{O}_{\chi,d\mu_v}(\mathbbm{1}_{\mathfrak{gl}_n(\mathcal{O}_{k_v})})=q_v^{S_v(\chi)}.
\]
\end{enumerate}
On the other hand, by \cite[Section 4.1]{Yun13}, we have $\rho_v(\chi)=S_v(\chi)-\sum\limits_{i\in B_v(\chi)}S_v(\chi_{v,i})$.
Since $S_v(\chi_{v,i})=0$ for a linear factor $\chi_{v,i}(x)\in\mathcal{O}_{k_v}[x]$ of $\chi(x)$, we verify that $$\rho_v(\chi)=
      \left\{
      \begin{array}{l l}
        0   & \textit{if }|B_v(\chi)|=1;\\
        S_v(\chi)-\delta_v  & \textit{if }|B_v(\chi)|=2;\\
        S_v(\chi)&\textit{if }|B_v(\chi)|=3.
      \end{array}
      \right.$$
    In order to write the above in a uniform way, we describe  $\mathrm{T}_{x_0, \mathcal{O}_{k_v}}(\kappa_v)$ explicitly as follows.
%We reformulate the above results using the geometric series and the following description of $T_{E_v}(\kappa_v)$ for each case.
\[\mathrm{T}_{x_0, \mathcal{O}_{k_v}}(\kappa_v)\cong \left\{\begin{array}{l l}
\kappa_{K_v}^\times & \textit{in the case (a), if $K_v/k_v$ is unramified;}\\
(\kappa_v[x]/(x^3))^\times& \textit{in the case (a), if $K_v/k_v$ is ramified;}\\
(\kappa_{K_{v,1}}\times\kappa_v)^\times& \textit{in the case (b), if $K_{v,1}/k_v$ is unramified;}\\
(\kappa_v[x]/(x^2)\times\kappa_v)^\times& \textit{in the case (b), if $K_{v,1}/k_v$ is ramified;}\\ 
(\kappa_v^3)^\times&\textit{in the case (c)},
\end{array}\right.
\]
where $\kappa_{K_v}$ is a cubic field extension over $\kappa_v$ in the first case, and $\kappa_{K_{v,1}}$ is a quadratic field extension over $\kappa_v$ in the third case.
Plugging it into the above formula for $\mathcal{O}_{\chi,d\mu_v}(\mathbbm{1}_{\mathfrak{gl}_n(\mathcal{O}_{k_v})})$, we obtain the desired formula.
\end{enumerate}
\end{proof}

\bibliographystyle{alpha}
\bibliography{References}

\end{document}